\documentclass[12pt]{amsart}
\usepackage{hyperref}
\usepackage[margin=1in]{geometry} 
\usepackage{amsmath,amsthm,amssymb,amsfonts}
\usepackage{comment}
\usepackage{mathtools}
\usepackage{dsfont}

\newtheorem{theorem}[]{Theorem}
\DeclareFontFamily{OMX}{yhex}{}
\DeclareFontShape{OMX}{yhex}{m}{n}{<->yhcmex10}{}
\DeclareSymbolFont{yhlargesymbols}{OMX}{yhex}{m}{n}
\DeclareMathAccent{\wideparen}{\mathord}{yhlargesymbols}{"F3}

\DeclareMathOperator{\arc}{arc}

\makeatletter
\newcommand{\newreptheorem}[2]{\newtheorem*{rep@#1}{\rep@title}\newenvironment{rep#1}[1]{\def\rep@title{#2 \ref*{##1}}\begin{rep@#1}}{\end{rep@#1}}}
\makeatother

\newtheorem{proposition}{Proposition}[]
\newtheorem{lemma}{Lemma}[]
\newtheorem{conjecture}{Conjecture}[]
\newtheorem{corollary}{Corollary}[]
\newreptheorem{theorem}{Theorem}
\newreptheorem{lemma}{Lemma}

\theoremstyle{definition}
\newtheorem{definition}{Definition}[section]
\newtheorem{remark}{Remark}[]

\usepackage{tikz}
\usepackage{tikz,tikz-cd}
\usetikzlibrary{calc}

\usetikzlibrary{patterns}

\definecolor{fuchsia}{RGB}{240,50,180}
\definecolor{spring}{RGB}{180,220,150}
\definecolor{lavender}{RGB}{200,140,220}
\definecolor{raddish}{RGB}{250,120,120}
\definecolor{tangerine}{RGB}{250,190,90}

\numberwithin{equation}{section}

\usepackage{caption}
\usepackage{subcaption}

\newcommand\polar[2]{({#1*cos(#2)},{#1*sin(#2)})}

\usepackage{float}

\newcommand\numeq[1]
{\stackrel{\scriptscriptstyle(\mkern-1.5mu#1\mkern-1.5mu)}{=}}
\newcommand\numstrict[1]
{\stackrel{\scriptscriptstyle(\mkern-1.5mu#1\mkern-1.5mu)}{<}}
\newcommand\numineq[1]
{\stackrel{\scriptscriptstyle(\mkern-1.5mu#1\mkern-1.5mu)}{\leq}}

  \author{Emily Casey}
\address{Emily Casey \newline
Department of Mathematics, University of Washington \newline
C138 Padelford Hall Box 354350, Seattle, WA 98195, USA}
\email{\href{mailto:ecasey4@uw.edu}{ecasey4@uw.edu}}
\begin{document}
	
	\title{Quantitative control on the Carleson $\varepsilon$-function determines regularity}
	\maketitle
 \begin{abstract}
     Carleson's $\varepsilon^2$-conjecture states that for Jordan domains in $\mathbb{R}^2$, points on the boundary where tangents exist can be characterized in terms of the behavior of the $\varepsilon$-function. This conjecture, which was fully resolved by Jaye, Tolsa, and Villa in 2021, established that qualitative control on the rate of decay of the Carleson $\varepsilon$-function implies the existence of tangents, up to a set of measure zero. We prove that quantitative control on the rate of decay of this function gives quantitative information on the regularity of the boundary.  
 \end{abstract}

	\section{Introduction}  
        In this paper we continue the study, which began in \cite{Bi87, BCGJ89, JTV21}, of the relationship between the boundary regularity of Jordan domains and the rate of decay of the Carleson $\varepsilon$-function. Geometric functions and their relationship to the regularity of sets has been the focus of many works over the past four decades, beginning in \cite{J90, DS91}. In particular, a geometric function measuring how well the boundary of a domain cuts circles in half appeared in \cite{Bi87}. This function is known as the Carleson $\varepsilon$-function. 
        \begin{definition}[Carleson $\varepsilon$-function]\label{epsilon}
  Let $\Omega^+\subset \mathbb{R}^2$ be  Jordan domain, i.e. a domain whose boundary is a Jordan curve, and let 
$\Omega^-=\mathbb{R}^2\setminus \overline{\Omega^+}$.
			Let $I^{\pm}(x,r)$ denote the largest open arc (with respect to $\mathcal{H}^1$-measure) on the circumference $\partial B(x,r)$ contained in $\Omega^{\pm}$. Define 
			\begin{equation*}
				\varepsilon(x,r)=\frac{1}{r}\max\left\{|\pi r-\mathcal{H}^1(I^+(x,r))|,|\pi r-\mathcal{H}^1(I^-(x,r))| \right\}.
			\end{equation*}
		\end{definition}
  \begin{figure}[h]
			\centering 
			\begin{subfigure}[]{0.37\textwidth}
				\centering
    \begin{tikzpicture}[scale=0.9]
     \draw (-1,3.9) -- (1.1,3.9) -- (1.1,2.7)--(-1,2.7)--(-1,3.9);
	\filldraw[lavender] (-0.9,3.8) -- (-0.5,3.8) -- (-0.5,3.4)--(-0.9,3.4)--(-0.9,3.8);
	\node at (0.3,3.6) {$I^+(x,r)$};
	\filldraw[spring] (-0.9,3.2) -- (-0.5,3.2) -- (-0.5,2.8)--(-0.9,2.8)--(-0.9,3.2);
	\node at (0.3,3) {$I^-(x,r)$};
    
	\node at (2.1,1.3) {\Large $\Omega^+$};
	\coordinate (A) at (0,.7); 
	\def\a{90}; 
	\coordinate (B) at (.9,1.7);
	\def\b{20};
	\coordinate (C) at (1.7,2.5);
	\def\c{100};
	\coordinate (D) at (2.1,2.7);
	\def\d{-10};
	\coordinate (E) at (2.3,2.8);
	\def\e{60};
	\coordinate (F) at (3,2.9);
	\def\f{-20};
	\coordinate (G) at (3.2,2.7);
	\def\g{-50};
	\coordinate (H) at (3.3,2.5);
	\def\h{-100};
	\coordinate (I) at (3.2,2);
	\def\i{-60};
	\coordinate (J) at (3.7,1.7);
	\def\j{10};
	\coordinate (K) at (4.7,1.8);
	\def\k{-60};
	\coordinate (L) at (3.5,1);
	\def\l{-180};
	\coordinate (M) at (3,.5);
	\def\m{-100};
	\coordinate (N) at (1.5,.2);
	\def\n{120};
	\coordinate (O) at (1,.3);
	\def\o{-140};
	
	\coordinate (X) at (1.355,1.905);
	\coordinate (Y) at (.45,1.48);
	\draw[line width=5pt, spring] (X) arc (25:205:.5);
	\draw[line width=5pt, lavender] (X) arc (25:-155:.5);
	
	\draw[thick] (B) circle (.5);
	\draw (X) circle (1.5pt);
	\draw (Y) circle (1.5pt);
	
	\draw[very thick] 
	(A) to [out=\a,in={\b-180}]
	(B) to [out=\b,in={\c-180}]
	(C) to [out=\c,in={\d-180}]
	(D) to [out=\d,in={\e-180}] 
	(E) to [out=\e,in={\f-180}] 
	(F) to [out=\f,in={\g-180}] 
	(G) to [out=\g,in={\h-180}] 
	(H) to [out=\h,in={\i-180}] 
	(I) to [out=\i,in={\j-180}] 
	(J) to [out=\j,in={\k-180}] 
	(K) to [out=\k,in={\l-180}] 
	(L) to [out=\l,in={\m-180}] 
	(M) to [out=\m,in={\n-180}] 
	(N) to [out=\n,in={\o-180}] 
	(O) to [out=\o,in={\a-180}] (A);
	
	\filldraw[fuchsia] (B) circle (1.5pt);
	\node[below=.5] at (B) {$x$};
	\end{tikzpicture}
				\begin{minipage}{.1cm}
					\vfill
				\end{minipage}\quad 
				 \subcaption{$\varepsilon(x,r)\approx 0$}\label{fig:1a}
			\end{subfigure} \qquad 
			\begin{subfigure}[]{0.37\textwidth}
				\centering
			    \begin{tikzpicture}[>=latex,scale=1]
	\node at (2.6,2.6) {\Large $\Omega^+$};
	\coordinate (A) at (0,.7); 
	\def\aa{-10};
	\def\ab{-20};
	\coordinate (B) at (1.4,2.7);
	\def\b{10};
	\coordinate (C) at (2,2.65);
	\def\c{10};
	\coordinate (D) at (2.8,3.2);
	\def\d{20};
	\coordinate (E) at (4.2,3.2);
	\def\e{-60};
	\coordinate (F) at (3.7,2.7);
	\def\f{170};
	\coordinate (G) at (3,2.6);
	\def\g{-100};
	\coordinate (H) at (3,1.95);
	\def\h{-60};
	\coordinate (I) at (3.2,1);
	\def\i{-90};
	\coordinate (J) at (2.4,0);
	\def\ja{10};
	\def\jb{30};
	\coordinate (K) at (2.9,1);
	\def\k{90};
	\coordinate (L) at (2.7,2);
	\def\la{-100};
	\def\lb{-110};
	\coordinate (M) at (2.6,1);
	\def\m{-60};
	\coordinate (N) at (2.2,-.1);
	\def\na{80};
	\def\nb{90};
	\coordinate (O) at (2.4,1);
	\def\o{120};
	\coordinate (P) at (2.4,2.05);
	\def\pa{-100};
	\def\pb{-140};
	\coordinate (Q) at (2,1.3);
	\def\q{-80};
	\coordinate (R) at (2,0);
	\def\ra{120};
	\def\rb{150};
	\coordinate (S) at (1.6,.5);
	\def\s{100};
	\coordinate (Sb) at (1.7,1.2);
	\def\sb{100};
	\coordinate (T) at (2.1,2.3);
	\def\ta{-170};
	\def\tb{-175};
	\coordinate (Ua) at (1.2,2.2);
	\def\ua{-130};
	\coordinate (Ub) at (.9,1.5);
	\def\ub{-110};
	\coordinate (X) at (1.78,.17);
	\coordinate (Ya) at (1.91,.3);
	\coordinate (Yb) at (2.33,.38);
	\coordinate (Yc) at (2.42,.35);
	\coordinate (Yd) at (2.62,.175);
	\coordinate (Ye) at (2.66,.078);
	\draw[line width=5pt, spring] (X) arc (148-360:21:.5);
	\draw[line width=5pt, lavender] (X) arc (148:126:.5);
	
	\draw[thick] (N) circle (.5);
	
	\foreach \x in {X,Ya,Yb,Yc,Yd,Ye}{\draw (\x) circle (1.5pt);}
	
	\draw[very thick] 
	(A) to [out=\aa,in={\b-180}]
	(B) to [out=\b,in={\c-180}]
	(C) to [out=\c,in={\d-180}]
	(D) to [out=\d,in={\e-180}] 
	(E) to [out=\e,in={\f-180}] 
	(F) to [out=\f,in={\g-180}] 
	(G) to [out=\g,in={\h-180}] 
	(H) to [out=\h,in={\i-180}] 
	(I) to [out=\i,in={\ja}] 
	(J);
	\draw[very thick] (J) to [out=\jb,in={\k-180}] (K) to [out=\k,in={\la}] 
	(L);
	
	\draw[very thick] (L) to [out=\lb,in={\m-180}] 
	(M) to [out=\m,in={\na}] (N);
	
	\draw[very thick] (N) to [out=\nb,in={\o-180}] 
	(O) to [out=\o,in={\pa}] (P);
	
	\draw[very thick] (P) to [out=\pb,in={\q-180}] 
	(Q) to [out=\q,in={\ra}] (R);
	
	\draw[very thick] (R) to [out=\rb,in={\s-180}] 
	(S) to [out=\s,in={\sb-180}] (Sb) to [out=\sb,in={\ta}] (T);
	
	\draw[very thick] (T) to [out=\ta,in={\ua-180}] (Ua)  to [out=\ua,in=\ub-180] (Ub) to [out=\ub,in={\ab}] (A); 
	
	\filldraw[fuchsia] (N) circle (1.5pt);
	\node[below=.5] at (N) {$x$};	
	\end{tikzpicture}
	\begin{minipage}{.1cm}
					\vfill
	\end{minipage}
	\subcaption{$\varepsilon(x,r)\approx \pi$}\label{fig:1b}
			\end{subfigure}
			\caption{}
		\end{figure}
        Carleson conjectured that for Jordan domains, points on the boundary where tangents exist could be characterized in terms of a Dini-type control on this function. 
        \begin{conjecture}[Carleson's $\varepsilon^2$-conjecture]\label{c: carleson's conjecture}
            Suppose $\Omega$ is a Jordan domain. Except for a set of zero $\mathcal{H}^1$-measure, $x$ is a tangent point of $\partial \Omega$ if and only if  
				\begin{equation}\label{intro epsilon}
					\int_0^1 \varepsilon(x,r)^2\frac{dr}{r}<\infty.
			\end{equation}
        \end{conjecture}
    The ``only if'' direction of the conjecture is well-known, see for instance \cite{Bi87,BJ94,GM05}. The converse direction was open for more than thirty years, only recently being resolved by Jaye, Tolsa, and Villa in \cite{JTV21}. 
   
   In \cite{J90}, Jones introduced another geometric function, known as the $\beta$-numbers, they measure how far a set deviates from a line inside of a ball.
   \begin{definition}[Jones $\beta$-numbers]\label{d: beta numbers}
    Let $\Gamma\subset \mathbb{R}^2$ be closed. For $x\in \Gamma$ and $r>0$, 
\begin{equation*}
    \beta(x,r)=\frac{1}{r}\inf_{L:\,L\cap B(x,r)\neq \emptyset}\sup_{y\in \Gamma\cap B(x,r)} d(y, L)\qquad \text{where $L$ is an affine line}.
\end{equation*}
\end{definition}
		
  Conjecture \ref{c: carleson's conjecture}, after being fully established by \cite{JTV21}, gives a characterization of the tangent points of Jordan domains up to a set of measure zero. In \cite{BJ94}, the authors proved a version of Conjecture \ref{c: carleson's conjecture} in terms of the $\beta$-numbers. Both \cite{JTV21} and \cite{BJ94} are results in the plane, so it is natural to consider the case of higher dimensions. In \cite{FTV23}, Fleschler, Tolsa, and Villa prove a higher dimensional analogue of Conjecture \ref{c: carleson's conjecture}. Additionally, in \cite{HV21}, Hyde and Villa prove a higher dimensional version of Conjecture \ref{c: carleson's conjecture} involving a variant of the $\beta$-numbers.
   
  In \cite{DKT01}, the authors prove that having quantitative control on the rate of decay of the $\beta$-numbers provides quantitative information about the regularity of $\partial \Omega$. We state this result for $1$-dimensional sets in $\mathbb{R}^2$, although the result in \cite{DKT01} is proved in higher dimensions.
 \begin{proposition}[Proposition 9.1, \cite{DKT01}]\label{prop: DKT}
			Let $0<\beta\leq 1$ be given. Suppose that $\Gamma$ is a \textit{Reifenberg-flat} set with vanishing constant of dimension $1$ in $\mathbb{R}^2$ (see definition \ref{def: Reifenberg flat}),
			and that for each compact set $K\subset \Gamma$ there exists a constant $C_K$ such that
			\begin{equation}\label{e: quant beta number decay}
				\beta(x,r)\leq C_K r^{\gamma}\qquad \text{for all}\quad x\in K, r\leq 1.
			\end{equation}
			Then $\Gamma$ is a $C^{1,\gamma}$-submanifold of dimension $1$ in $\mathbb{R}^2$.
\end{proposition}
Many authors have shown that the behavior of the $\beta$-numbers determines the regularity of a given set, see for example \cite{J90,DS91,O92,DS93,BJ94,T95,P97,DKT01,DT12,AT15,Tolsa2015,ghinassi2017sufficient,AS18,HV21,del2022geometric}, and many others.
		In this paper, we address a natural question, namely, does quantitative control of the Carleson $\varepsilon$-function along the lines of Proposition \ref{prop: DKT} yield higher regularity of the boundary. Our main result is the following theorem.

		\begin{theorem}\label{THRM: main theorem}
			Let $\Omega^+\subseteq \mathbb{R}^2$ be a Jordan domain and $\Omega^-:=\mathbb{R}^2\setminus \overline{\Omega^+}$. Let $\Gamma:= \partial \Omega^{\pm}$. Fix some $\alpha\in (0,1)$ and $C_0>0$. Suppose there exists $r_0>0$ such that 
			\begin{equation}\label{eq: epsilon infinity}
				\varepsilon(x,r)\leq C_0r^{\alpha} \qquad \text{for all} \qquad x\in \Gamma, \, r\leq r_0, 
			\end{equation}
			then $\Gamma$ is a $C^{1, \beta}$-manifold of dimension $1$ in $\mathbb{R}^2$, where $\beta$ depends on $\alpha$.
		\end{theorem}
		
		In Proposition \ref{prop: DKT}, the authors assume an additional hypothesis, that the set is a Reifenberg-flat set with vanishing constant, which we omit. In \cite{DKT01}, the sets considered are more general, and this assumption ensures that $\Gamma$ has no holes. 
		\begin{definition}[Bilateral $\beta$-numbers]
		  For $x\in \Gamma$ and $r>0$, set   
      \begin{equation}\label{e: bilateral beta}
			b\beta(x,r)=\frac{1}{r}\inf_{L \in G(1,2)}\,D[\Gamma\cap B(x,r), (L+x)\cap B(x,r)],
		\end{equation}
		where 
		\begin{equation*}
			D[E,F]=\sup_{y\in E}d(y,F)+\sup_{y\in F}d(y,E)
		\end{equation*}
		and $G(1,2)$ denotes the set of all $1$-dimensional subspaces of $\mathbb{R}^2$.
		\end{definition}

		\begin{definition}[Reifenberg-flat set with vanishing constant]\label{def: Reifenberg flat}
			Let $\delta>0$. A closed set $\Gamma\subset \mathbb{R}^2$ is $\delta$-Reifenberg-flat of dimension $1$ if for all compact sets $K\subset \Gamma$ there is a radius $r_K>0$ such that 
			\begin{equation*}
				b\beta(x,r)\leq \delta 
    \qquad \text{for all} \qquad x\in K, \, r\leq r_K.
			\end{equation*}
			$\Gamma\subset \mathbb{R}^2$ is a \emph{Reifenberg-flat set with vanishing constant} if for every compact set $K\subset \Gamma$, 
			\begin{equation*}
				\lim_{r\to 0^+}\sup_{x\in K}b\beta(x,r)=0.
			\end{equation*}
		\end{definition}

While $\beta(x,r)$ gives control on $\beta(y,s)$ for $y\in B(x,r/2)$ and $s\in (r/4,r/2)$ the same is not true for the Carleson $\varepsilon$-function, and it is therefore not stable enough to apply the techniques that have been used to relate the behavior of the $\beta$-numbers to the regularity of a given set. See for instance the discussion in \cite[p.141]{DS91}. To emphasize the difference between the $\varepsilon$-function and the $\beta$-numbers, we give the following example. Suppose $\Gamma$ is an exponential spiral. At the origin $\varepsilon(x,r)=0$ but $\beta(x,r)\approx r$, since $\varepsilon$-function takes a measurement of the boundary on circular shells and the $\beta$-numbers take a measurement of the boundary inside balls. See Figure \ref{fig:2}. It is true that if $x$ is a tangent point for $\Gamma$, then $\beta(x,r)$ controls $\varepsilon(x,s)$ up to a multiplicative constant for $s\in (\frac{r}{2},r)$, see \cite{JTV21}. 

		\begin{figure}[h]
			\centering 
				\begin{tikzpicture}[scale=1.6]
					\def\angl{180}
					
					\draw[domain=0:2200, samples=200, smooth, variable=\x, ultra thick]  plot ({.002*exp(.003*\x)*cos(\x)}, {.002*exp(.003*\x)*sin(\x)}); 
					\draw[domain=0:2225, samples=200, smooth, variable=\x, ultra thick, black]  plot ({.002*exp(.003*\x)*cos(\x+\angl)}, {.002*exp(.003*\x)*sin(\x+\angl)}); 
					
					\coordinate (X) at (0,0); 
					\node (Y) at  \polar{2.12}{0} {}; 
					\node[left=35] at (X) {$\Omega^+$};
					\node (Z) at (1,0) {};

					\foreach \x in {300,1700,2150}{
						\draw[blue] (0,0) circle ({.002*exp(.003*\x)}); 
						\filldraw[black] ({.002*exp(.003*\x)*cos(\x+\angl)}, {.002*exp(.003*\x)*sin(\x+\angl)}) circle (1pt); 
						
						\filldraw[black] ({.002*exp(.003*\x)*cos(\x)}, {.002*exp(.003*\x)*sin(\x)}) circle (1pt); 
					}
					
				\end{tikzpicture}
				\caption{$\varepsilon(x,r)= 0\quad\beta(x,r)\approx r$}\label{fig:2}
		\end{figure}
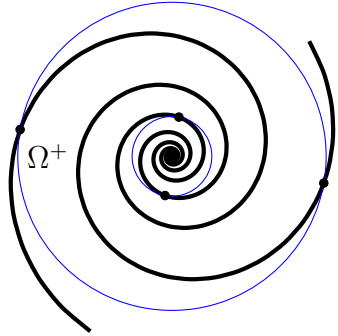

		To prove Theorem \ref{THRM: main theorem}, we show that under the hypotheses, $\Gamma$ is a Reifenberg-flat set with vanishing constant and \eqref{e: quant beta number decay} holds with $\beta=\alpha/16$. Then Proposition \ref{prop: DKT} gives our result. With the discussion above in mind, we remark that the majority of the work in this paper is in proving that the hypothesis on the $\varepsilon$-function in \eqref{eq: epsilon infinity} gives control on the $\beta$-numbers as in \eqref{e: quant beta number decay}. We prove this control, along with the vanishing Reifenberg-flatness, in two lemmas.	
  
		\begin{lemma}\label{LEM: Main Lemma 1}
			Let $\Omega^+\subseteq \mathbb{R}^2$ be a Jordan domain and $\Omega^-:=\mathbb{R}^2\setminus \overline{\Omega^+}$. Let $\Gamma:= \partial \Omega^{\pm}$. Fix some $\alpha\in (0,1)$ and $C_0>0$. Suppose there exists $r_0>0$ such that 
			\begin{equation*}
				\varepsilon(x,r)\leq C_0r^{\alpha} \qquad \text{for all} \qquad x\in \Gamma, \, r\leq r_0.
			\end{equation*}
            Then, there exists $\rho_0\leq r_0$ so that for $x_0\in \Gamma$ and $r\leq\rho_0$ there exists a line $L_0$ through $x_0$ such that 
			\begin{equation}
				\sup_{x\in L_0\cap B(x_0,r)} d(x,\Gamma\cap B(x_0,r))< Cr^{\alpha/2+1}.
			\end{equation}
		\end{lemma}

		\begin{lemma}\label{LEM: Main Lemma 2}
			Let $\Omega^+\subseteq \mathbb{R}^2$ be a Jordan domain and  $\Omega^-:=\mathbb{R}^2\setminus \overline{\Omega^+}$. Let $\Gamma:= \partial \Omega^{\pm}$. Fix some $\alpha\in (0,1)$ and $C_0>0$. Suppose there exists $r_0>0$ such that 
			\begin{equation*}
				\varepsilon(x,r)\leq C_0r^{\alpha} \qquad \text{for all} \qquad x\in \Gamma, \, r\leq r_0.
			\end{equation*}
			Then, there exists $\rho_1\leq r_0$ so that 
            \begin{equation*}
                \beta(x_0,r)\leq Cr^{\alpha/16} \qquad \text{for}\quad x_0\in \Gamma \quad \text{and}\quad r\leq \rho_1.
            \end{equation*}
		\end{lemma}
        Note that the exponents $\alpha/2$ in Lemma \ref{LEM: Main Lemma 1} and $\alpha/16$ in Lemma \ref{LEM: Main Lemma 2} are not optimal, and could be improved by making the constants larger. This method will never yield that the decay of the $\varepsilon$-function matches the regularity of the boundary.

		Lemma \ref{LEM: Main Lemma 1} and Lemma \ref{LEM: Main Lemma 2} together yield that $\Gamma$ satisfies the flatness hypothesis necessary to apply Proposition \ref{prop: DKT}. Moreover, Lemma \ref{LEM: Main Lemma 2} gives the necessary decay on the $\beta$-numbers. The proof of Theorem \ref{THRM: main theorem} is then straightforward. Thus, $\Gamma$ is a $C^{1, \frac{\alpha}{16}}$-submanifold of dimension $1$ in $\mathbb{R}^2$ by Proposition \ref{prop: DKT}.
		
		In Section \ref{section: prelims}, we prove a general consequence of the assumption on $\varepsilon(x,r)$. In Section \ref{section: Proof of lemma 1} we prove Lemma \ref{LEM: Main Lemma 1} in several steps, each step its own lemma. In Section \ref{section: proof of lemma 2} we prove Lemma \ref{LEM: Main Lemma 2}.

    \section*{Acknowledgments}
		The author was partially supported by DMS Focus Research Grant 1853993. The author would like to thank Tatiana Toro for being a source of constant support and guidance, Max Goering for his feedback on early drafts of this paper, and Sarafina Ford for her generosity and expertise in tikz.
	\section{Preliminaries}\label{section: prelims}
		Since $\varepsilon(x,t)<C_0t^{\alpha}$, for all $x\in \Gamma$ and $t\leq r_0$, for each such $x$ and $t$ there exist (connected) arcs $I^{\pm}(x,t)\subset \Omega^{\pm}$ such that  
		\begin{equation}\label{eq:bounds on lengths of I+-}
			\begin{array}{rll}
				\pi t-C_0t^{\alpha+1}&\leq  \mathcal{H}^1(I^+(x,t))&\leq  \pi t +C_0 t^{\alpha+1},\\
				\pi t-C_0t^{\alpha+1}&\leq  \mathcal{H}^1(I^-(x,t))&\leq  \pi t +C_0 t^{\alpha+1}. 
			\end{array}
		\end{equation}
		
		Thus the situation on $\partial B(x,t)$ is very rigid, in the sense that the portion of this circumference that is not covered by either $I^+(x,t)$ or $I^-(x,t)$ has length at most $2C_0t^{\alpha+1}$, and since $I^{\pm}(x,t)$ are connected, the location where $\Gamma$ intersects this circumference is tightly controlled. We make this idea rigorous in the next lemma.
		
		\begin{figure}[h]
			\centering
			\begin{tikzpicture}[scale=1]
				\def\angl{160}
				
				\draw[line width = 5pt, spring]  \polar{2}{\angl}  arc(\angl:0:2);
				
				\draw[line width = 5pt, tangerine]  \polar{2}{\angl}  arc(\angl:200:2);
				
				\coordinate (X) at (0,0); 
				\draw[ultra thick] (X) circle (2);
				\filldraw[black] (X) circle (2pt);
				\node[below right=2] at (X) {$x$};

				\node (Y) at  \polar{2}{\angl} {}; 
				\filldraw[black](Y) circle (2pt);

				\node (V) at  \polar{2}{180} {};
				
				\filldraw[black](2,0) circle (2pt);
				
				\draw (-3,0)--(3,0); 
				
				\node (Xr) at  \polar{2}{-\angl} {}; 
				\filldraw[black](Xr) circle (2pt);
				
				\node (W) at  \polar{2}{160} {}; 
				
				\node (V1) at  \polar{3}{160} {};
				\node (V2) at  \polar{3}{-\angl} {};
				\draw[dashed] (X)--(Y);
				\draw[dashed] (X)--(Xr);
				
				\node (A) at  (-0.4,0.4) {}; 
				\node at  (A) {$C_0t^{\alpha}$};
				\draw[thick]  \polar{1}{\angl}  arc(\angl:180:1);
				\draw (A)--(-0.75,0.05);
				\draw[thick]  \polar{0.9}{180}  arc(180:200:0.9);
				\node (D) at  (-0.4,-0.4) {}; 
				\node at  (D) {$C_0t^{\alpha}$};
				\draw (D)--(-0.75,-0.05);
				
				\node (B) at \polar{2.5}{45} {};
				\node[above] at (B) {$I^-(x,t)$};

				\draw (-5,-3) -- (5.5,-3) -- (5.5,-3.7)--(-5,-3.7)--(-5,-3);
				\filldraw[tangerine] (-4.9,-3.1) -- (-4.5,-3.1) -- (-4.5,-3.5)--(-4.9,-3.5)--(-4.9,-3.1);
				\node at (0.5,-3.3) {\text{possible locations for the second endpoint of $I^-(x,t)$}};
				
			\end{tikzpicture}
		\end{figure}
		\begin{figure}[H]
			\begin{center}  
				\begin{tikzpicture}[scale=1]
					\def\angl{160}
					
					\draw[line width = 5pt, spring]  \polar{2}{\angl}  arc(\angl:0:2);
					\draw[line width = 5pt, lavender]  \polar{2}{200}  arc(200:360:2);
					
					\coordinate (X) at (0,0); 
					\draw[ultra thick] (X) circle (2);
					\filldraw[black] (X) circle (2pt);
					\node[below right=2] at (X) {$x$};

					\node (Y) at  \polar{2}{\angl} {}; 
					\filldraw[black](Y) circle (2pt);

					\node (V) at  \polar{2}{180} {};
					
					\filldraw[black](2,0) circle (2pt);
					
					\draw (-3,0)--(3,0); 
					
					\node (Xr) at  \polar{2}{-\angl} {}; 
					\filldraw[black](Xr) circle (2pt);
					
					\node (W) at  \polar{2}{160} {}; 
					
					\node (V1) at  \polar{3}{160} {};
					\node (V2) at  \polar{3}{-\angl} {};
					\draw[dashed] (X)--(Y);
					\draw[dashed] (X)--(Xr);
					
					\node (A) at  (-0.4,0.4) {}; 
					\node at  (A) {$C_0t^{\alpha}$};
					\draw[thick]  \polar{1}{\angl}  arc(\angl:180:1);
					\draw (A)--(-0.75,0.05);
					\draw[thick]  \polar{0.9}{180}  arc(180:200:0.9);
					\node (D) at  (-0.4,-0.4) {}; 
					\node at  (D) {$C_0t^{\alpha}$};
					\draw (D)--(-0.75,-0.05);
					
					\node (B) at \polar{2.5}{45} {};
					\node[above] at (B) {$I^-(x,t)$};
					\node (C) at \polar{3}{-45} {};
					\node[above] at (C) {$I^+(x,t)$};
				\end{tikzpicture}
				\caption{Possible configuration}
				\label{fig: trapped boundary}
			\end{center}
		\end{figure}
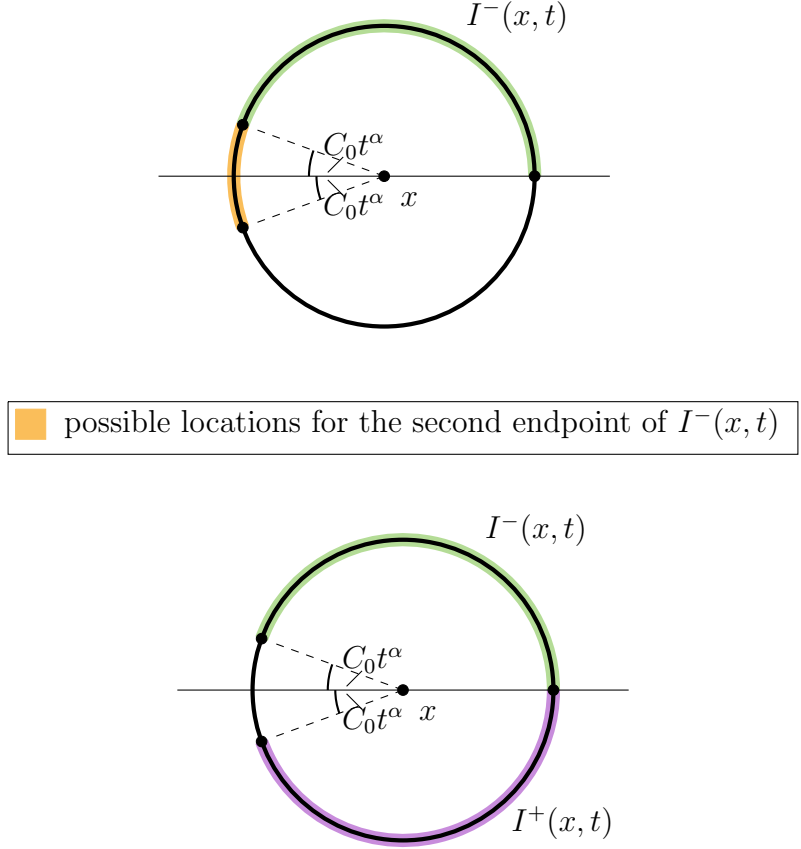

		\begin{lemma}[Trapped Boundary Lemma]\label{LEM: Trapped boundary lemma}
			Assume the hypotheses of Theorem \ref{THRM: main theorem}. Fix $x\in \Gamma$ and $t\leq r_0$. If $v,w\in \partial B(x,t)\cap \Gamma$, let $\arc_{x,t}(v,w)$ denote the shortest arc on $\partial B(x,t)$ with endpoints $v$ and $w$. Then 
			\begin{equation}\label{eq: where can the bdry live on circular shell}
				\mathcal{H}^1(\arc_{x,t}(v,w))\in [0, 2C_0 t^{\alpha+1}] \cup [\pi t- C_0t^{\alpha+1}, \pi t].
			\end{equation}
		\end{lemma}
		\vspace{\baselineskip}
		\begin{proof}
		Since $x,t$ are fixed throughout the proof, we drop the dependence on them from the notation. 
		We have the following two cases: 
		\begin{enumerate}
			\item[(i)] $\arc(v,w)\cap I^+=\emptyset$ and $\arc(v,w)\cap I^-=\emptyset$, or 
			\item[(ii)] either $\arc(v,w)\cap I^+\neq \emptyset$ or $\arc(v,w)\cap I^-\neq \emptyset$. 
		\end{enumerate}
		
		If 
		$\arc(v,w)\cap I^+=\emptyset$ and $\arc(v,w)\cap I^-=\emptyset$, then, 
		\begin{equation*}
			2\pi t\geq \mathcal{H}^1(\arc(v,w))+\mathcal{H}^1(I^+)+\mathcal{H}^1(I^-)\geq 2\pi t-2C_0t^{\alpha+1}+\mathcal{H}^1(\arc(v,w)).
		\end{equation*}
		It then follows immediately  
		\begin{equation*}
			0\leq \mathcal{H}^1(\arc(v,w))\leq 2C_0t^{\alpha+1}.
		\end{equation*}
		
		Now suppose that $\arc(v,w)\cap I^+\neq \emptyset$. Since $v,w\in \Gamma$, $I^+$ is the largest connected component on $\partial B(x,t)\cap \Omega^+$ and $\arc(v,w)$ connected, it follows that $I^+\subset \arc(v,w)$, and  
		\begin{equation*}
			\mathcal{H}^1(\arc(v,w))\geq \mathcal{H}^1(I^+)\geq \pi t -C_0t^{\alpha+1}.
		\end{equation*}
		Since $\arc(v,w)$ is the shortest arc between $v$ and $w$ on $\partial B(x,t)$, 
		\begin{equation*}
			\mathcal{H}^1(\arc(v,w))\leq \pi t.
		\end{equation*}
	\end{proof}
 
		Lemma \ref{LEM: Trapped boundary lemma} is crucial in the proofs of both Lemma \ref{LEM: Main Lemma 1} and Lemma \ref{LEM: Main Lemma 2} and also yields the following corollary.
		
		\begin{corollary}[Line Choice]\label{COR: line choice} Suppose $y,w\in \partial B(x_0,t)\cap \Gamma$. Let $L_t(x_0,w)$ denote the line determined by $x_0$ and $w$, and let $L_t(x_0,y)$ be the line determined by $x_0$ and $y$. Then, 
			\begin{equation}\label{eq: line choice lem conclusion}
				\measuredangle(L_t(x_0,w), L_t(x_0,y))\leq 2C_0t^{\alpha}.
			\end{equation}
		\end{corollary}
		\noindent\textbf{\textit{Proof.}} Let $y,w\in \partial B(x_0,t)\cap \Gamma$. Then either $\mathcal{H}^1(\arc_{x_0,t}(y,w))\in [0,2C_0t^{\alpha+1}]$ or $\mathcal{H}^1(\arc_{x_0,t}(y,w))\in [\pi t -c_0t^{\alpha+1}, \pi t]$ by Lemma \ref{LEM: Trapped boundary lemma}. If $\mathcal{H}^1(\arc_{x_0,t}(y,w))\in [0,2C_0t^{\alpha+1}]$, then 
        \begin{equation*}
            \measuredangle(L_t(x_0,w), L_t(x_0,y))=\frac{\mathcal{H}^1(\arc_{x_0,t}(y,w))}{t}\leq 2C_0t^{\alpha}. 
        \end{equation*}
        If $\mathcal{H}^1(\arc_{x_0,t}(y,w))\in [\pi t -C_0t^{\alpha+1}, \pi t]$, let $y^{\prime}=2x_0-y$ be the antipodal point to $y$ on $\partial B(x_0,t)$. Then $\mathcal{H}^1(\arc_{x_0,t}(y^{\prime},w))\in [0,C_0t^{\alpha+1}]$ and $L(x_0,y)=L(x_0,y^{\prime})$, so as above,
        \begin{equation*}
            \measuredangle(L_t(x_0,w), L_t(x_0,y))\leq 2C_0t^{\alpha}.
        \end{equation*}
		\begin{flushright}
			$\square$
		\end{flushright}
		\begin{definition}[Good Approximating Line]\label{def: good approx line}
			Fix $x\in \Gamma$ and $r<r_0$. For any $y\in \partial B(x,r)\cap \Gamma$, we call $L_r(x,y)$  \textbf{\textit{a good approximating line}} at scale $r$ for $\Gamma$ at $x$. 
		\end{definition}

	\section{Proof of Lemma \ref{LEM: Main Lemma 1}}\label{section: Proof of lemma 1}
		To prove Lemma \ref{LEM: Main Lemma 1}, it is sufficient to find a dyadic collection of points on the line $L_0$, which we denote $L_0^0$ for coherence in the argument, so that for each point in the collection, there is a boundary point sufficiently close by. The idea, is that as long as the dyadic collection can be taken at a fine enough scale, then every point in the line is close to a dyadic point, which is, in turn, close to the boundary. We begin by showing that at the first dyadic scale, scale $\frac{1}{2}$, there is a boundary point close by. 
		
		\begin{lemma}[Half-Scales Lemma]\label{approx at single point} 
			Under the hypotheses of Lemma \ref{LEM: Main Lemma 1} there exists $\rho_0\leq r_0$ such that for $x_0\in \Gamma$ and any $s_0\leq \rho_0$ the following holds. Let $y\in \partial B(x_0,s_0)\cap \Gamma$, and let $L_0^0:=L_{s_0}(x_0,y)$. Let $x_1^1(L_0^0)$ be the midpoint of the segment between $x_0$ and $y$ on $L_0^0$, that is 
			\begin{equation*}
				|x_1^1(L_0^0)-x_0|=\frac{1}{2^1}s_0.
			\end{equation*}
			Then, there exists a point $z_1^1\in \Gamma$ such that 
			\begin{equation}\label{radius close to halfway}
				|x_1^1(L_0^0)-z_1^1|<C_0s_0^{\alpha/2+1},
			\end{equation}
			and 
			\begin{equation}\label{e: radius for half-scales}
				|s_1-\frac{s_0}{2}|<C_0^2s_0^{\alpha+1} \quad \text{where}\quad s_1=|z_1^1-x_0|=|z_1^1-y|.\\
			\end{equation}
			Moreover,
			\begin{align}\label{small angles between halfway lines}
				\begin{split}
					\measuredangle(L_0^0, L_1^0)&\leq C_1s_0^{\alpha/2}\\
					\measuredangle(L_0^0, L_1^1)&\leq C_1s_0^{\alpha/2}
				\end{split}
			\end{align}
			where $L_1^{0}$ denotes the line through $x_0$ and $z_1^1$, and $L_1^{1}$ denotes the line through $z_1^1$ and $y$, and $C_1$ is a constant depending only on $C_0$.
		\end{lemma}
		\begin{remark}\label{halfway distance from line to boundary is small}
			Observe that the conclusions of Lemma \ref{approx at single point}, tell us that
			\begin{equation}\label{halfway dist to Gamma}
				d(x_1^1(L_0^0), \Gamma\cap B(x_0,s_0))\leq C_0s_0^{\alpha/2+1}. 
			\end{equation}
		\end{remark}
		\begin{figure}[H]
			\begin{center}
				\begin{tikzpicture}[scale=2]
					\def\angl{7}
					
					\coordinate (X1) at (0,0); 
					\filldraw[black] (X1) circle (2pt);
					\node[below=2] at (X1) {$x_1^1(L_0)$};

					\coordinate (Xr) at (4,0); 
					\filldraw[black] (Xr) circle (2pt);
					\node[below=2] at (Xr) {$y$};
					
					\coordinate (X0) at (-4,0); 
					\filldraw[black] (X0) circle (2pt);
					\node[below=2] at (X0) {$x_0$};
					
					\draw (X0)--(Xr); 
					
					\coordinate (Z) at (0,2); 
					\filldraw[black] (Z) circle (2pt);
					\node[above=2] at (Z) {$z_1^1$};

					\draw (Z)--(X0);
					\draw (Z)--(Xr);
					
					\draw (X0) to node[midway, above] {$s_1\quad L_1^0$}(Z);
					\draw (Z) to node[midway, above] {$L_1^1\quad s_1$}(Xr);
					\draw (X0) to node[midway, below] {$\frac{s_0}{2}$}(X1);
					\draw (X1) to node[midway, below] {$\frac{s_0}{2}$}(Xr);
					\draw[dashed] (Z) to node[midway, right] {$<C_0s_0^{\alpha/2+1}$}(X1);
					
					\draw (.2,0)--(0.2,.2);
					\draw (0,.2)--(0.2,0.2);
				\end{tikzpicture}
                \caption{}
				\label{fig: conclusion of half scales lemma}
			\end{center}
		\end{figure}
	
   \begin{proof} Let $x_0\in \Gamma$ and suppose $s_0\leq \rho_0$, for $\rho_0\leq r_0$ to be chosen later. Assume (by rotating and translating if necessary) $L_0^0$ is the horizontal axis, and $y\in (L_0^0)^+$. Consider $B(x_0,s_0)$ and $B(y,s_0)$. These balls intersect at two points, say $a\in \mathbb{R}^2_+$ and $b\in \mathbb{R}^2_-$. Since 
        \begin{equation}\label{e: s_0 1 half-scales}
           2C_0s_0^{\alpha+1}<\mathcal{H}^1(\arc_{y,s_0}(x_0,a))=\frac{\pi}{3}s_0<\pi s_0-C_0s_0^{\alpha+1}, 
        \end{equation}
        for $s_0<\left(\frac{\pi}{6C_0}\right)^{\frac{1}{\alpha}}$, we can conclude that $a,b$ are not boundary points from the Trapped Boundary Lemma. The same holds for $\mathcal{H}^1(\arc_{x_0,s_0}(a,y))$. Thus, we can assume without loss of generality that $a\in \Omega^+\cap \mathbb{R}^2_+$ and $b \in \Omega^- \cap \mathbb{R}^2_-$.
		\tikzset{every picture/.style={line width=0.75pt}} 
		
		\tikzset{every picture/.style={line width=0.75pt}} 
		
		\tikzset{every picture/.style={line width=0.75pt}} 
		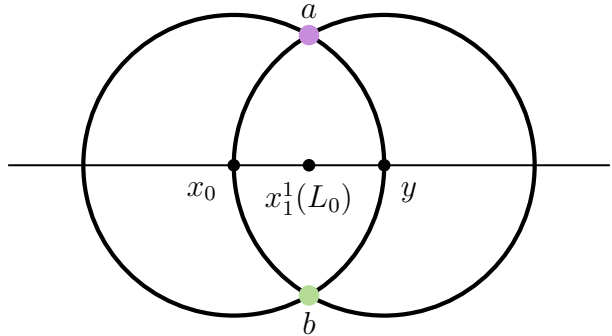
\begin{figure}[H]
			\begin{tikzpicture}[scale=1]
				\def\angl{7}
				
				
				\coordinate (X) at (0,0); 
				\filldraw[black] (X) circle (2pt);
				\node[below=2] at (X) {$x_1^1(L_0)$};
				
				\def\shft{1}
				\def\eps{.6}
				
				
				\coordinate (Xr) at (\shft,0); 
				\draw[ultra thick] (Xr) circle ({2*\shft});
				\filldraw[black] (Xr) circle (2pt);
				\node[below right=2] at (Xr) {$y$};
				
				\coordinate (X0) at (-\shft,0); 
				\draw[ultra thick] (X0) circle ({2*\shft});
				\filldraw[black] (X0) circle (2pt);
				\node[below left=2] at (X0) {$x_0$};
				
				\draw ({4*\shft},0)--({-4*\shft},0); 
				
				\coordinate (a) at (0,{2*\shft*cos(30)}); 
				\filldraw[lavender] (a) circle (3.5pt);
				\node[above=2] at (a) {$a$};
				
				\coordinate (b) at (0,{-2*\shft*cos(30)});
				\filldraw[spring] (b) circle (3.5pt);
				\node[below=2] at (b) {$b$};
				
				
				
			\end{tikzpicture}
			
			\ \caption{The set up}
			\label{fig: set up of half scales}
			
		\end{figure}
		
		Thus there must be exist some point, $z_1^1\in \overline{ab}\cap \Gamma$. We claim that this point $z_1^1$ is ``close'' to $L_0^0$, and in particular, that $z_1^1 \in (L_0^0)_{C_0s_0^{\alpha/2+1}}$, where 
		\begin{equation*}
			(L_0^0)_{C_0s_0^{\alpha/2+1}}:=\{x\in \mathbb{R}^2: d(x,L_0^0)<C_0s_0^{\alpha/2+1}\}.
		\end{equation*}
		
		Suppose not. That is,
		\begin{equation}\label{eq:distance z_0 to L(r)}
			d:=d(z_1^1,L_0^0)\geq C_0s_0^{\alpha/2+1}. 
		\end{equation}
		
		Consider $B(z_1^1,s_1)$, where $s_1=|z_1^1-x_0|=|z_1^1-y|$. To contradict the Trapped Boundary Lemma we will show 
		\begin{equation}\label{eq: contradiction in half scales}
			2C_0s_1^{\alpha+1}<\mathcal{H}^1(\arc_{z_1^1,s_1}(x_0,y))<\pi s_1-C_0 s_1^{\alpha+1}.
		\end{equation}
		Let $\gamma$ be the angle between the radius $\overline{z_1^1x_0}$ and the line segment $\overline{z_1^1b}$. 
		
		\begin{figure}[H]
			\begin{center}
				\begin{tikzpicture}[scale=1]
					\def\angl{7}
					
					
					\coordinate (X1) at (0,0); 
					
					\def\shft{1}
					\def\eps{.6}
					
					\draw[<->] ({4*\shft},\eps) to node[midway, right] {$C_0{s_0^{\alpha/2+1}}$}({4*\shft},0);
					
					\coordinate (Xr) at (\shft,0); 
					\filldraw[black] (Xr) circle (2pt);
					\node[below right=2] at (Xr) {$y$};
					
					\coordinate (X0) at (-\shft,0); 
					\filldraw[black] (X0) circle (2pt);
					\node[below left=2] at (X0) {$x_0$};
					
					\draw ({4*\shft},0)--({-4*\shft},0); 
					\draw[dashed] ({4*\shft},\eps)--({-4*\shft},\eps); 
					\draw[dashed]  ({4*\shft},-\eps)--({-4*\shft},-\eps); 
					
					\coordinate (Z) at (0,{\shft*cos(30)}); 
					\filldraw[black] (Z) circle (2pt);
					\node[left=2] at (Z) {$z_1^1$};

					\def\rad{sqrt((\shft)^2+(\shft*cos(30))^2)}
					\draw[ultra thick] (Z) circle (\rad);
					\draw[line width=2pt, raddish]  (X0) arc({180+40.9}:{360-40.9}:1.323*\shft);
					
					\draw (Z)--(X0);
					\draw (Z)--(Xr);
					\draw (Z)--(X1);
					
					\coordinate (P) at (-.01,{\shft*cos(30)-0.2});
					\node[below left=1] at (P) {$\mathbf{\gamma}$};
					\coordinate (Q) at (.005,{\shft*cos(30)-0.2});
					\node[below right=2] at (Q) {$\mathbf{\gamma}$};
					
					\draw (0,0.2)--(0.2,0.2); 
					\draw (0.2,0.2)--(0.2,0);

				\end{tikzpicture}
				
				\caption{$\arc_{z_1^1,s_1}(x_0,y)$ has measure $2s_1\gamma$.}
				\label{fig: z11 has to be close}
			\end{center}
		\end{figure}
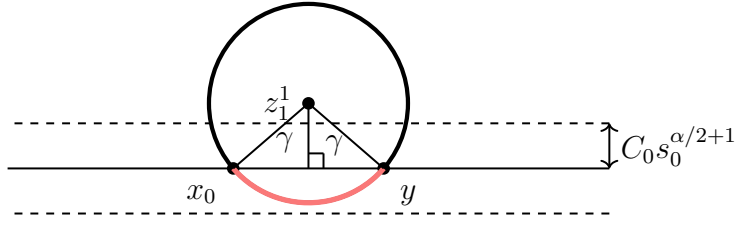
		
		Observe that 
		\begin{align*}
			\cos\,\gamma\,&=\frac{d}{s_1}
			\geq \frac{C_0s_0^{\alpha/2+1}}{s_1} >C_0s_0^{\alpha/2},  
		\end{align*}
		where the first inequality follows from (\ref{eq:distance z_0 to L(r)}) and the last from the fact that $s_0/2\leq s_1<s_0$. Moreover, since the furthest $z_1^1$ can be from $L_0^0$ is $\frac{\sqrt{3}}{2}s_0$, we also have $\cos\, \gamma<\frac{\sqrt{3}}{2}$. Thus,
		$$\frac{\pi }{6}<\gamma < \arccos\,C_0s_0^{\alpha/2}.$$
		Using the Taylor series for $\arccos(x)$ for $x=C_0s_0^{\alpha/2}$ near $0$ 
		we see that
		\begin{align*}
			\arccos(C_0s_0^{\alpha/2})
			&< \frac{\pi}{2}-C_0s_0^{\alpha/2}.
		\end{align*}
		
		Since $\mathcal{H}^1(\arc_{z_1^1,s_1}(x_0,y))=2\gamma s_1$, it follows that 
		\begin{align}\label{e: upper bound on arc in half scales}
			\begin{split}
				\frac{\pi}{3}s_1<\mathcal{H}^1(\arc_{z_1^1,s_1}(x_0,y))&<2\arccos(C_0s_0^{\alpha/2})s_1\\
				&<2\left(\frac{\pi}{2}-C_0s_0^{\alpha/2}\right)s_1\\
				&< \pi s_1-C_0s_1^{\alpha/2}s_1\\
				&< \pi s_1-C_0s_1^{\alpha+1},\\
			\end{split}
		\end{align}
		where we use the fact that, since $s_0<1$, $s_1^{\alpha}<s_1^{\alpha/2}<s_0^{\alpha/2}$ in the final two inequalities. Thus we have obtained the right-hand side inequality in (\ref{eq: contradiction in half scales}). To obtain the left-hand side inequality, it is only necessary to take $s_0$ sufficiently small (which in turn means that $s_1$ is also sufficiently small). Indeed, for $s_0<\left( \frac{\pi}{6C_0}\right)^{\frac{1}{\alpha}}$, since $s_1<s_0$, 
		\begin{equation}\label{e: lower bound on arc in half scales}
			2C_0s_1^{\alpha}s_1<2C_0\left(\left( \frac{\pi}{6C_0}\right)^{\frac{1}{\alpha}}\right)^{\alpha}s_1=\frac{\pi}{3}s_1.
		\end{equation}
		Thus \eqref{e: upper bound on arc in half scales} and \eqref{e: lower bound on arc in half scales} yield (\ref{eq: contradiction in half scales}), which contradicts the Trapped Boundary Lemma for the boundary point $z_1^1$ at scale $s_1$. Thus, we conclude that $d(z_1^1, L_0^0)<C_0 s_0^{\alpha/2+1}$, as claimed. 
		
		Let $\theta=\measuredangle(L_0^0,L_1^0)=\measuredangle(L_0^0,L_1^1)$. Since $\tan (\theta)=\frac{d}{s_0/2}<\frac{C_0s_0^{\alpha/2+1}}{s_0/2}=2C_0 s_0^{\alpha/2}$, we have that 
		\begin{equation}\label{eq: bound on theta Taylor series}
			\theta<\arctan(2C_0s_0^{\alpha/2}).
		\end{equation}
		Using the Taylor expansion for $\arctan$, it is clear that for $x$ small enough 
		\begin{equation*}
			|\arctan\,x-x|\leq \frac{|x|^3}{3}, 
		\end{equation*}
		and thus for $s_0$ small enough,
		\begin{equation}\label{e: half scales angle bound}
			|\arctan(2C_0s_0^{\alpha/2})-(2C_0s_0^{\alpha/2})|\leq \frac{1}{3}(2C_0s_0^{\alpha/2})^3\leq 2C_0^2s_0^{\alpha/2}.
		\end{equation}
		
		Thus \eqref{e: half scales angle bound} and (\ref{eq: bound on theta Taylor series}) yield
		\begin{equation*}
			\theta<2C_0s_0^{\alpha/2}+2C_0^2s_0^{\alpha/2}=C_1s_0^{\alpha/2}, 
		\end{equation*}
		where $C_1=2C_0+2C_0^2$. This gives us (\ref{small angles between halfway lines}).

		It is left to show \eqref{e: radius for half-scales}. Since
		\begin{equation*}
			(s_0/2)^2+d^2=s_1^2\quad  \text{and} \quad s_1\geq s_0/2, 
		\end{equation*}
		we have 
		\begin{align*}
			|s_1-s_0/2|
			&=\frac{d^2}{s_1+s_0/2}
			< \frac{(C_0s_0^{\alpha/2+1})^2}{s_1+s_0/2}
			\leq \frac{C_0^2s_0^{\alpha+2}}{s_0} 
			=C_0^2s_0^{\alpha+1},
		\end{align*}
		since $s_0/2\leq s_1\leq s_0$. Choosing $\rho_0$ small enough so that \eqref{e: s_0 1 half-scales},\eqref{e: upper bound on arc in half scales},\eqref{e: lower bound on arc in half scales}, and \eqref{e: half scales angle bound} hold completes the proof.
   \end{proof}
   
		\begin{remark}
			For $s_1$ in Lemma \ref{approx at single point}, note that 
			\begin{equation*}
				\frac{s_0}{2}\leq s_1\leq \frac{s_0}{2}+C_0^2s_0^{\alpha+1}.
			\end{equation*}
		\end{remark}

		The result of the Half-Scales Lemma gives the following set up. The points, $x_0$,$y$ on $L_0^0$ and $z_1^1$, constructed above, form an isosceles triangle with height, base angles, and side lengths controlled by the hypothesis on the $\varepsilon(x,r)$-function in (\ref{eq: epsilon infinity}).

    In Lemma \ref{LEM: Step 1} we iteratively apply the Half-Scales Lemma to find points $z_k^j$ in $\Gamma$ at the appropriate scale.
		
		\begin{lemma}\label{LEM: Step 1} Under the hypotheses of Lemma \ref{LEM: Main Lemma 1}, for $k\in \mathbb{N}$ there exists a collection of points $\{z_k^j\}_{j=0}^{2^k}\subset \Gamma$ such that for each $1\leq j \leq 2^k-1$ odd, $T_k^j:=z_{k}^{j-1}z_k^jz_{k}^{j+1}$ is an isosceles triangle with side lengths 
			\begin{equation}\label{e: the side lengths are equal}
				\left|z_{k}^{j-1}-z_k^j\right|=\left|z_{k}^{j+1}-z_k^j\right| \qquad \text{and} \qquad  \left|z_{k}^{j-1}-z_{k}^{j+1} \right|,
			\end{equation}
			satisfying
			\begin{align}\label{eq: side length relationship on triangle}
				\begin{split}
					\left|\left|z_{k}^{j-1}-z_k^j\right|-\frac{\left|z_{k}^{j-1}-z_{k}^{j+1} \right|}{2} \right|<C_0^2\left|z_{k}^{j-1}-z_{k}^{j+1} \right|^{\alpha+1}, \\
					\\
					\left|\left|z_{k}^{j+1}-z_k^j\right|-\frac{\left|z_{k}^{j-1}-z_{k}^{j+1} \right|}{2} \right|<C_0^2\left|z_{k}^{j-1}-z_{k}^{j+1} \right|^{\alpha+1}.
				\end{split}
			\end{align}
            Moreover, for $j$ even, $z_k^j=z_{k-1}^{j/2}$.
            Denote the sides of $T_k^j$ as follows, 
            \begin{equation*}
                L_k^{j-1}:=\overline{z_k^{j-1}z_k^j}, \quad L_k^j:= \overline{z_k^jz_k^{j+1}}, \quad \text{and} \quad L_{k-1}^{\frac{j-1}{2}}:=\overline{z_{k-1}^{\frac{j-1}{2}}z_{k-1}^{\frac{j+1}{2}}}.
            \end{equation*}
            Then the base angles of $T_k^j$,  $\measuredangle L_{k-1}^{\frac{j-1}{2}}L_k^{j-1}$ and $\measuredangle L_{k-1}^{\frac{j-1}{2}}L_k^j$ are equal, and both angles satisfy the following bound:
			\begin{equation}\label{eq: angles for triangle }
				\begin{split}
					&\measuredangle\left(L_k^{j-1}L_{k-1}^{\frac{j-1}{2}}\right)<C_1 \left|z_{k}^{j-1}-z_{k}^{j+1} \right|^{\alpha/2}
					\quad \text{and} \quad
					\measuredangle\left(L_k^jL_{k-1}^{\frac{j-1}{2}}\right)<C_1 \left|z_{k}^{j-1}-z_{k}^{j+1} \right|^{\alpha/2}.
				\end{split}
			\end{equation}
			
      For $m\in \mathbb{N}$ and $1\leq n \leq 2^m-1$, denote by $x_m^n(L_{k}^{j})$ the point on $L_k^j$ such that
		\begin{equation}\label{e: dyadic points}
			|x_m^n(L_{k}^{j})-z_k^j|=\frac{n}{2^m}|z_k^{j+1}-z_k^j|, 
		\end{equation}
		where $|z_k^{j+1}-z_k^j|$ is the length of $L_k^j$. Then, for any $m\in \mathbb{N}$
			\begin{equation}\label{eq: vertex of triangle always close to midpoint}
				|z_k^j-x_m^{2^{m-1}}(L_{k-1}^{\frac{j-1}{2}})|<C_0\left|z_{k}^{j-1}-z_{k}^{j+1} \right|^{\alpha/2+1}.
			\end{equation}

		\end{lemma}    
		\begin{figure}[H]
			\begin{center}
				\begin{tikzpicture}[scale=1.5]
					
					\coordinate (X1) at (0,0); 
					\filldraw[black] (X1) circle (2pt);
					\node[below=2] at (X1) {$x_2^2(L_{k-1}^{\frac{j-1}{2}})$};
					\coordinate (X2) at (-2,0); 
					\filldraw[black] (X2) circle (2pt);
					\node[below=2] at (X2) {$x_2^1(L_{k-1}^{\frac{j-1}{2}})$};
					\coordinate (X3) at (2,0); 
					\filldraw[black] (X3) circle (2pt);
					\node[below=2] at (X3) {$x_2^3(L_{k-1}^{\frac{j-1}{2}})$};
					
					\coordinate (Xr) at (4,0); 
					\filldraw[black] (Xr) circle (2pt);
					\node[below=2] at (Xr) {$z_{k}^{j+1}=z_{k-1}^{\frac{j+1}{2}}$};
					
					\coordinate (w1) at (-2,1); 
					\filldraw[black] (w1) circle (2pt);
					\node[below right=2] at (w1) {\small{$x_1^1(L_k^{j-1})$}};
					\coordinate (w2) at (2,1); 
					\filldraw[black] (w2) circle (2pt);
					\node[below=2] at (w2) {$x_1^1(L_k^{j})$};
					
					\coordinate (X0) at (-4,0); 
					\filldraw[black] (X0) circle (2pt);
					\node[below=2] at (X0) {$z_k^{j-1}=z_{k-1}^{\frac{j-1}{2}}$};
					
					\draw (X0)--(Xr); 
					
					\coordinate (Z) at (0,2); 
					\filldraw[black] (Z) circle (2pt);
					\node[above=2] at (Z) {$z_k^j$};

                \coordinate (T) at (0,1); 
				\node[] at (T) {$T_k^j$};
                \coordinate (S1) at (-2,1); 
				\node[above left] at (S1) {$L_k^{j-1}$};
                \coordinate (S2) at (2,1); 
				\node[above right] at (S2) {$L_k^{j}$}; \coordinate (B) at (0,-.7); 
				\node[below=2] at (B) {$L_{k-1}^{\frac{j-1}{2}}$};   
					\draw (Z)--(X0);
					\draw (Z)--(Xr);

				\end{tikzpicture}
				\caption{Example: Dyadic points on the lines $L_k^{j-1}$, $L_k^{j}$, and $L_{k-1}^{\frac{j-1}{2}}$}
				\label{fig: dyadic labeling on base of triangle}
			\end{center}
		\end{figure}
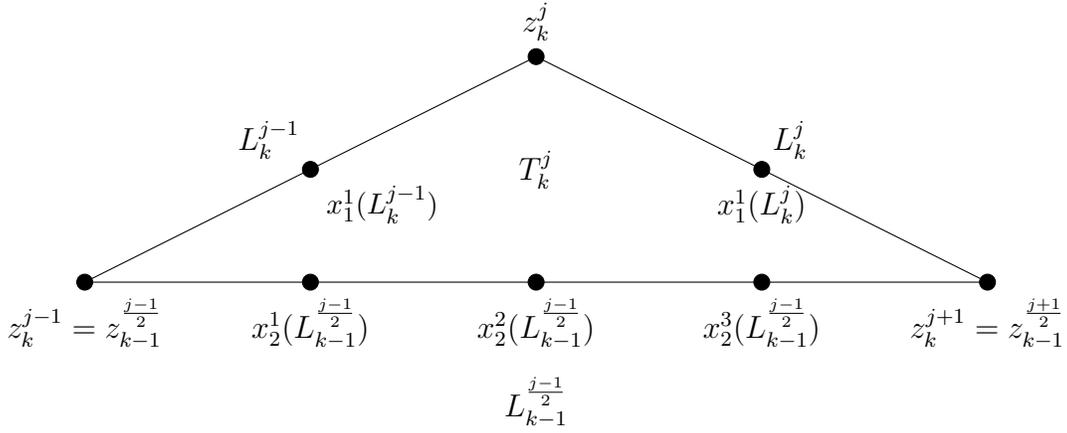

    After proving Lemma \ref{LEM: Step 1}, the goal is to then prove that each $z_k^j$ is very close to a dyadic point, $x_k^j(L_0^0)$, in $L_0^0$. This is proved in two steps. In Lemma \ref{LEM: Triangle Lemma} we prove that dyadic points on the base of any triangle constructed in the Half-Scales Lemma is close to an ``intermediate'' dyadic point on one of the sides of that triangle. In Lemma \ref{LEM: Iterated Triangle Lemma (Step 3)} we iterate Lemma \ref{LEM: Triangle Lemma} to prove that the dyadic point we started with on $L_0^0$, $x_k^j(L_0^0)$, is close to the midpoint of $L_{k-1}^{\frac{j-1}{2}}$, $x_1^1(L_{k-1}^{\frac{j-1}{2}})$, which we know is close to $z_k^j$ from the Half-Scales Lemma.   
		\begin{remark}
			Note that this notation is consistent with the Half-Scales Lemma, where we chose $x_1^1(L_0^0)$ to be the point such that $|x_1^1(L_0^0)-z_0^0|=\frac{1}{2^1}s_0$, where $z_0^0=x_0$, $y=z_0^1$, and $s_0=|z_0^1-z_0^0|$.
		\end{remark}

  \begin{proof}[Proof of Lemma \ref{LEM: Step 1}]
  We build this collection inductively. Let $x_0$, $y$, and $s_0$, be as in the statement of the Half-Scales Lemma. The case $k=1$ follows from the Half-Scales Lemma with 
  \begin{equation*}
      T_1^1=z_1^0z_1^1z_1^2, \quad L_0^0=\overline{z_0^0z_0^1}, \quad L_1^0=\overline{z_1^0z_1^1}, \quad L_1^1=\overline{z_1^1z_1^2},
  \end{equation*}
  where $z_1^0:=x_0$ and $z_1^2:=y$.
	
		\begin{figure}[H]
			\begin{center}
				\begin{tikzpicture}[scale=1.5]
					\def\angl{7}
					
					\coordinate (X1) at (0,0);			\coordinate (Xr) at (4,0); 
					\filldraw[black] (Xr) circle (2pt);
					\node[below right] at (Xr) {$y=z_0^1=z_1^2$};
					
					\coordinate (X0) at (-4,0); 
					\filldraw[black] (X0) circle (2pt);
					\node[below left] at (X0) {$x_0=z_0^0=z_1^0$};
					
					\draw (X0)--(Xr); 
					
					\coordinate (Z) at (0,2); 
					\filldraw[black] (Z) circle (2pt);
					\node[above=2] at (Z) {$z_1^1$};

					\draw (Z)--(X0);
					\draw (Z)--(Xr);
					
					\draw (X0) to node[midway, above] {$L_1^0$}(Z);
					\draw (Z) to node[midway, above] {$L_1^1$}(Xr);
					\draw (X0) to node[midway, below] {$L_0^0$}(Xr);
					\coordinate (T) at (0,1);
					\node at (T) {$T_1^1$};

				\end{tikzpicture}
				\caption{$k=1$}
			\end{center}
		\end{figure}
		\vspace{\baselineskip}
		
		Suppose for some $k$ there exists a collection of points, $\{z_k^j\}_{j=0}^{2^k}$, for which the conclusion of Lemma $\ref{LEM: Step 1}$ holds.

		\begin{figure}[H]
			\begin{center}
				\begin{tikzpicture}[scale=1.5]
					\def\angl{7}
					
					\coordinate (X1) at (0,0); 
					

					\coordinate (Xr) at (4,0); 
					\filldraw[black] (Xr) circle (2pt);
					\node[below=2] at (Xr) {$z_k^{j+1}=z_{k+1}^{2j+2}$};
					
					\coordinate (X0) at (-4,0); 
					\filldraw[black] (X0) circle (2pt);
					\node[below=2] at (X0) {$z_k^j=z_{k+1}^{2j}$};
					
					\draw (X0)--(Xr); 
					
					\coordinate (Z) at (0,2); 
					\filldraw[black] (Z) circle (2pt);
					\node[above=2] at (Z) {$z_{k+1}^{2j+1}$};

					\draw (Z)--(X0);
					\draw (Z)--(Xr);
					
     \filldraw[black] (0,0) circle (2pt);
					\node[below=2] at (0,0) {$x_1^1(L_k^j)$};
					\coordinate (T) at (0,1);
					\node at (T) {$T_{k+1}^{2j+1}$};
					
				\end{tikzpicture}
				\caption{\small{Relabeling the collection for $z_k^j=z_{k+1}^{2j}$ for $j=0,\hdots, 2^k$ }}
				\label{fig: relabeling}
			\end{center}
		\end{figure}
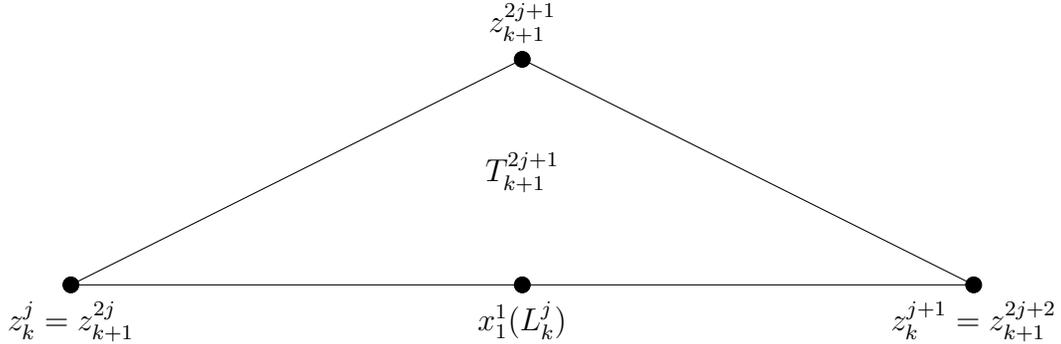
		
		We now show that we can find a collection $\{z_{k+1}^j\}_{j=1}^{2^{k+1}}\subset \Gamma$ of points so that Lemma \ref{LEM: Step 1} holds. We begin by relabeling $\{z_k^j\}_{j=0}^{2^k}$,  
		\begin{equation*}
			z_k^j= z_{k+1}^{2j} \qquad \text{for}\quad j=0,\hdots,2^k, 
		\end{equation*}
		see Figure \ref{fig: relabeling}, for example. Now we construct $z_{k+1}^{2j+1}$ for each $0\leq j\leq 2^k-1$. Apply the Half-Scales Lemma to $L_k^j$, the line segment with endpoints $z_{k+1}^{2j}$ and $z_{k+1}^{2j+2}$, to obtain a point $z_{k+1}^{2j+1}\in \Gamma$ such that $T_{k+1}^{2j+1}:=z_{k+1}^{2j}z_{k+1}^{2j+1}z_{k+1}^{2j+2}$, which denotes the triangle formed by $z_{k+1}^{2j}$, $z_{k+1}^{2j+2}$, and $z_{k+1}^{2j+1}$, is an isosceles triangle with side lengths
		\begin{equation*}
			|z_{k+1}^{2j}-z_{k+1}^{2j+1}|=|z_{k+1}^{2j+1}-z_{k+1}^{2j+2}|\quad \text{and} \quad \left|z_{k+1}^{2j}-z_{k+1}^{2j+2}\right|,
		\end{equation*}
		such that 
		\begin{align*}
			\begin{split}
				& \left||z_{k+1}^{2j}-z_{k+1}^{2j+1}|-\frac{\left|z_{k+1}^{2j}-z_{k+1}^{2j+2}\right|}{2}\right|<C_0^2\left|z_{k+1}^{2j}-z_{k+1}^{2j+2}\right|^{\alpha+1}\\
				\\
				& \left||z_{k+1}^{2j+1}-z_{k+1}^{2j+2}|-\frac{\left|z_{k+1}^{2j}-z_{k+1}^{2j+2}\right|}{2}\right|<C_0^2\left|z_{k+1}^{2j}-z_{k+1}^{2j+2}\right|^{\alpha+1}.
			\end{split}
		\end{align*}
		Moreover, 
		\begin{equation*}
			\measuredangle(L_{k+1}^{2j}L_k^j)=\measuredangle(L_{k+1}^{2j+1}L_{k}^j)<C_1\left|z_{k+1}^{2j}-z_{k+1}^{2j+2}\right|^{\alpha/2}.
		\end{equation*}
		\begin{figure}[H]
			\begin{center}
				\begin{tikzpicture}[scale=1.5]
					
					\coordinate (X1) at (0,1); 
					\node at (X1) {$T_{k}^{j+1}$};
					\coordinate (X2) at (0,0); 
					\node[below] at (X2) {$L_{k-1}^{\frac{j}{2}}$};

					\coordinate (Xr) at (4,0); 
					\filldraw[black] (Xr) circle (2pt);
					\node[below=2] at (Xr) {$z_{k-1}^{\frac{j+2}{2}}=z_k^{j+2}=z_{k+1}^{2j+4}$};
					
					\coordinate (w1) at (-2,1); 
					\node[above=2] at (w1) {\small{$T_{k+1}^{2j+1}$}};

					\coordinate (w2) at (2,1); 
					
					\coordinate (X0) at (-4,0); 
					\filldraw[black] (X0) circle (2pt);
					\node[below=2] at (X0) {$z_{k-1}^{\frac{j}{2}}=z_k^j=z_{k+1}^{2j}$};
					
					\draw (X0)--(Xr); 
					
					\coordinate (Z) at (0,2); 
					\filldraw[black] (Z) circle (2pt);
					\node[above=2] at (Z) {$z_k^{j+1}=z_{k+1}^{2j+2}$};
					
					\coordinate (Z21) at (-2.5,2); 
					\filldraw[black] (Z21) circle (2pt);
					\node[left=2] at (Z21) {$z_{k+1}^{2j+1}$};
						
					\draw (Z)--(X0);
					\draw (Z)--(Xr);
					\draw[dashed] (Z21)--(X0);
					\draw[dashed] (Z21)--(Z);

					\draw [dashed] (X0) -- (Z21) node [midway, above left=2.5] {$L_{k+1}^{2j}$};
                    \draw [dashed] (Z21) -- (Z) node [midway, above=2] {$L_{k+1}^{2j+1}$};
                    \node[below=2] at (w1) {$L_{k}^j$};
                    \node[below=2] at (w2) {$L_{k}^{j+1}$};
				\end{tikzpicture}
				\caption{Construction of $z_{k+1}^{2j+1}$}
				\label{fig: two generations}
			\end{center}
		\end{figure}
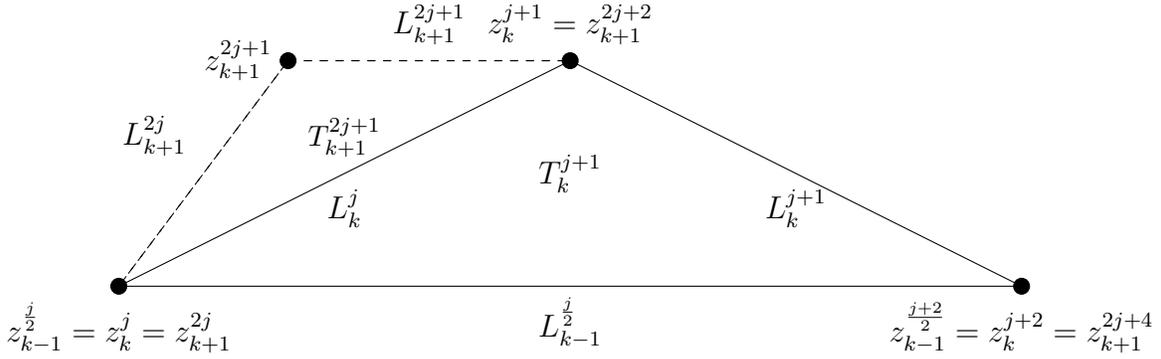
		Since \eqref{eq: vertex of triangle always close to midpoint} follows directly from the Half-Scales Lemma, the lemma is complete.
  \end{proof}
		
	It may seem worrisome that the bounds in the previous lemma depend on $j$, but for a fixed $k$, the $\left|z_{k}^{j-1}-z_{k}^{j} \right|$ are uniformly bounded in $j$.
		\begin{remark}\label{rmk: sk's} 
			Observe that at generation $k$ of the construction, there are $2^{k-1}$ applications of the Half-Scales Lemma. In generation $k$, for each $1\leq j\leq 2^{k}-1$ odd, there are possibly different values of $\left|z_{k}^{j-1}-z_{k}^{j} \right|=\left|z_{k}^{j}-z_{k}^{j+1} \right|$. Let $C_0^2s_0^{\alpha}<1/4$. Then, for $k,j=1$ by \eqref{e: radius for half-scales}
            \begin{equation}\label{eq: rmk 4 k=1}
                |z_1^1-z_1^0|<s_0\left(\frac{1}{2}+C_0^2s_0^{\alpha}\right)<\frac{3}{4}s_0.
            \end{equation}
            Let $k=2$. Then for $j=1$, from \eqref{eq: side length relationship on triangle} and \eqref{eq: rmk 4 k=1},
            \begin{equation*}
                |z_2^0-z_2^1|<|z_2^0-z_2^2|\left(\frac{1}{2} +C_0^2|z_2^0-z_1^1|^{\alpha}\right)< \frac{3}{4}s_0\left(\frac{1}{2} +C_0^2s_0^{\alpha}\right)<s_0\left(\frac{3}{4}\right)^2.
            \end{equation*}
            From \eqref{eq: side length relationship on triangle}, \eqref{e: the side lengths are equal}, and \eqref{eq: rmk 4 k=1}, 
            \begin{equation*}
                |z_2^2-z_2^3|<s_0\left(\frac{3}{4}\right)^2.
            \end{equation*}
            Continuing inductively gives the following uniform bound in $j$, 
			\begin{equation}\label{e: uniform bound s_k}
				\left|z_{k}^{j-1}-z_{k}^{j} \right|\leq s_0\left(\frac{3}{4}\right)^k \qquad \text{for all}\quad 1\leq j\leq 2^{k-1}\, \text{odd}.
			\end{equation}
		\end{remark}

		\begin{lemma}\label{LEM: Triangle Lemma} Assume the hypotheses of Lemma \ref{LEM: Main Lemma 1}. For $k\in \mathbb{N}$ and each $1\leq j\leq 2^k-1$ odd, let $T_{k}^j:= z_k^{j-1}z_k^{j}z_k^{j+1}$ denote the isosceles triangle from Lemma \ref{LEM: Step 1}. Take the dyadic collection of points on each side of $T_k^j$: 
			\begin{equation*}
				\{x_{m-1}^n(L_k^{j-1})\}_{n=1}^{2^{m-1}-1},\quad \{x_{m-1}^n(L_k^{j})\}_{n=1}^{2^{m-1}-1}, \quad \text{and}\quad \{x_m^n(L_{k-1}^{\frac{j-1}{2}})\}_{n=1}^{2^m-1}. 
			\end{equation*}
			
			Fix $m\in \mathbb{N}$. Then for each $1\leq n\leq 2^m-1$,
			\begin{align}\label{eq:dyadic points close one step}
				\begin{split}
                &|z_k^j-x_m^{2^{m-1}}(L_{k-1}^{\frac{j-1}{2}})|<C_0s_0^{\alpha/2+1}\left(\frac{3}{4} \right)^{(k-1)(\alpha/2+1)}\, \text{if} \quad n=2^{m-1}, \\
					&|x_m^n(L_{k-1}^{\frac{j-1}{2}})-x_{m-1}^{n}(L_{k}^{j-1})|<C_1s_0^{\alpha/2+1}\left(\frac{3}{4}\right)^{k+(k-1)\frac{\alpha}{2}}+C_0^2s_0^{\alpha+1}\left(\frac{3}{4}\right)^{(k-1)(\alpha+1)}\, \text{if} \quad n<2^{m-1}, \\
					\\
					&|x_m^n(L_{k-1}^{\frac{j-1}{2}})-x_{m-1}^{d}(L_{k+1}^{j})|<C_1s_0^{\alpha/2+1}\left(\frac{3}{4}\right)^{k+(k-1)\frac{\alpha}{2}}+C_0^2s_0^{\alpha+1}\left(\frac{3}{4}\right)^{(k-1)(\alpha+1)}\, \text{if} \,\, 2^{m-1}<n<2^{m},
				\end{split}
			\end{align}
		where $d=n-2^{m-1}$. 
	\end{lemma}
	\begin{figure}[H]
		\begin{center}
			\begin{tikzpicture}[scale=1]
				
				\coordinate (X1) at (0,0); 
				\filldraw[black] (X1) circle (2pt);
				\node[below=2] at (X1) {$x_1^1(L_{k-1}^{\frac{j-1}{2}})$};
				
				\coordinate (Xr) at (4,0); 
				\filldraw[black] (Xr) circle (2pt);
				\node[below right] at (Xr) {$z_k^{j+1}=z_{k-1}^{\frac{j+1}{2}}$};
				
				\coordinate (w1) at (-2,1); 
				\filldraw[black] (w1) circle (2pt);
				\node[above left=2] at (w1) {$x_1^1(L_k^{j-1})$};
				
				\coordinate (w2) at (2,1); 
				\filldraw[black] (w2) circle (2pt);
				\node[above right=2] at (w2) {$x_1^1(L_k^{j})$};

				\coordinate (X0) at (-4,0); 
				\filldraw[black] (X0) circle (2pt);
				\node[below left] at (X0) {$z_k^{j-1}=z_{k-1}^{\frac{j-1}{2}}$};
				
				\draw (X0)--(Xr); 
				
				\coordinate (Z) at (0,2); 
				\filldraw[black] (Z) circle (2pt);
				\node[above=2] at (Z) {$z_k^j$};

				\coordinate (P) at (0,1); 
				\node[] at (P) {$T_k^j$};

				\coordinate (X21) at (-2,0); 
				\filldraw[black] (X21) circle (2pt);
				\node[below=2] at (X21) {$x_2^1(L_{k-1}^{\frac{j-1}{2}})$};
				
				\coordinate (X23) at (2,0); 
				\filldraw[black] (X23) circle (2pt);
				\node[below=2] at (X23) {$x_2^3(L_{k-1}^{\frac{j-1}{2}})$};

				\draw (Z)--(X0);
				\draw (Z)--(Xr);

			\end{tikzpicture}
		\end{center}
        \caption{}
	\end{figure}

 \begin{proof}[Proof of Lemma \ref{LEM: Triangle Lemma}]
 Fix $k$ and $1\leq j\leq 2^k-1$ odd. Recall, the relabeling from Lemma \ref{LEM: Step 1}, $z_{k-1}^{\frac{j-1}{2}}=z_k^{j-1}$ and $z_{k-1}^{\frac{j+1}{2}}=z_k^{j+1}$. Let 
	\begin{equation*}
		s_k^j=\left|z_k^{j-1}-z_k^j\right|=\left|z_k^{j}-z_k^{j+1}\right| \qquad \text{and} \qquad s_{k-1}^j:=\left|z_k^{j-1}-z_k^{j+1}\right|.
	\end{equation*}
	Let $\theta_k^j=\measuredangle(L_k^{j-1}L_{k-1}^{\frac{j-1}{2}})=\measuredangle(L_k^{j}L_{k-1}^{\frac{j-1}{2}})$.
Fix $m\in \mathbb{N}$. First suppose that $n=2^{m-1}$. Then from \eqref{eq: vertex of triangle always close to midpoint} and \eqref{e: uniform bound s_k}, 
    \begin{equation*}
	|z_k^j-x_m^{2^{m-1}}(L_{k-1}^{\frac{j-1}{2}})|<C_0\left|z_{k}^{j-1}-z_{k}^{j+1} \right|^{\alpha/2+1}=C_0\left|z_{k-1}^{\frac{j-1}{2}}-z_{k-1}^{\frac{j+1}{2}} \right|^{\alpha/2+1}<C_0s_0^{\alpha/2+1}\left(\frac{3}{4}\right)^{(k-1)(\alpha/2+1)}.
    \end{equation*}
    
		\begin{figure}[H]
		\begin{center}
			\includegraphics[height=6.2cm]{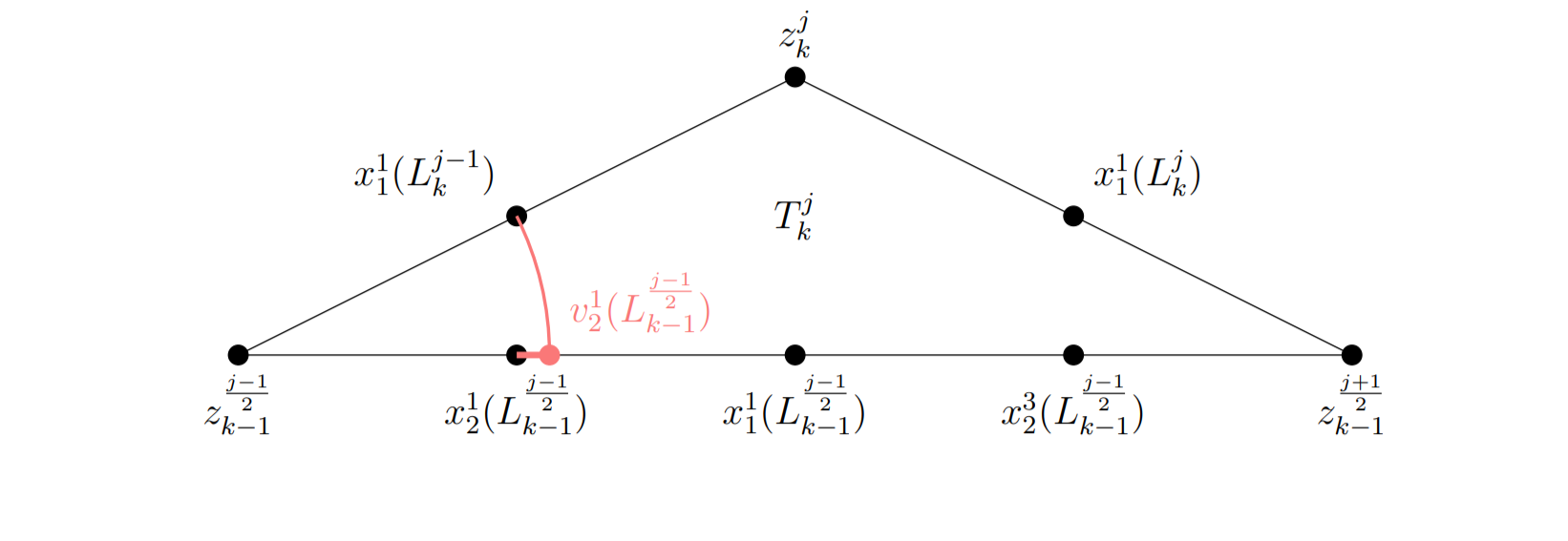}
			\caption{The case where $n<2^m$}
			\label{fig: Triangle lemma n<2^m}
			
		\end{center}
	\end{figure}
    Suppose now that $n<2^{m-1}$. Let $x_{m-1}^n(L_k^{j-1})$ be as in \eqref{e: dyadic points} and let $v_{m}^n(L_{k-1}^{\frac{j-1}{2}})$ be the point in $L_{k-1}^{\frac{j-1}{2}}$ such that 
	\begin{equation*}
		\left|v_{m}^n(L_{k-1}^{\frac{j-1}{2}})-z_k^{j-1}\right|=\frac{n}{2^{m-1}}s_k^j.
	\end{equation*}
	
	See Figure \ref{fig: Triangle lemma n<2^m}. Then, 
	\begin{align*}
		|x_{m-1}^n(L_{k}^{j-1})-v_{m}^n(L_{k-1}^{\frac{j-1}{2}})|&\leq \mathcal{H}^1(\arc_{z_k^{j-1}, \frac{n}{2^{m-1}}s_k^j}(x_{m-1}^n(L_{k}^{j-1}),v_{m}^n(L_{k-1}^{\frac{j-1}{2}})))\\
		&< \frac{n}{2^{m-1}}s_{k}^j\theta_{k}^j\\
		&< \frac{n}{2^{m-1}}s_{k}^j(C_1(s_{k-1}^j)^{\alpha/2})\\
		&=C_1\frac{n}{2^{m-1}}s_{k}^j(s_{k-1}^j)^{\alpha/2}, 
	\end{align*}
	where the second to last inequality follows from (\ref{eq: angles for triangle }).
	Furthermore, 
	\begin{align*}
		|v_{m}^n(L_{k-1}^{\frac{j-1}{2}})-x_{m}^n(L_{k-1}^{\frac{j-1}{2}})|&=\left|\frac{n}{2^{m-1}}s_{k}^j-\frac{n}{2^{m}}s_{k-1}^j\right|\\
		&=\frac{n}{2^{m-1}}\left|s_{k}^j-\frac{s_{k-1}^j}{2} \right|\\
		&<\frac{n}{2^{m-1}}C_0^2(s_{k-1}^j)^{\alpha+1}, 
	\end{align*}
	where the last inequality follows from (\ref{eq: side length relationship on triangle}). Thus from the triangle inequality we have
	\begin{equation*}
		|x_{m}^n(L_{k-1}^{\frac{j-1}{2}})-x_{m-1}^n(L_{k}^{j-1})|<C_1\frac{n}{2^{m-1}}s_{k}^j(s_{k-1}^j)^{\alpha/2}+\frac{n}{2^{m-1}}C_0^2(s_{k-1}^j)^{\alpha+1},
	\end{equation*}
	which together with \eqref{e: uniform bound s_k} gives (\ref{eq:dyadic points close one step}). Observe that the case where $n>2^{m-1}$ is similar, but we include the proof below for completeness.  \\
	
    	\begin{figure}[H]
		\begin{center}
			\includegraphics[height=6.2cm]{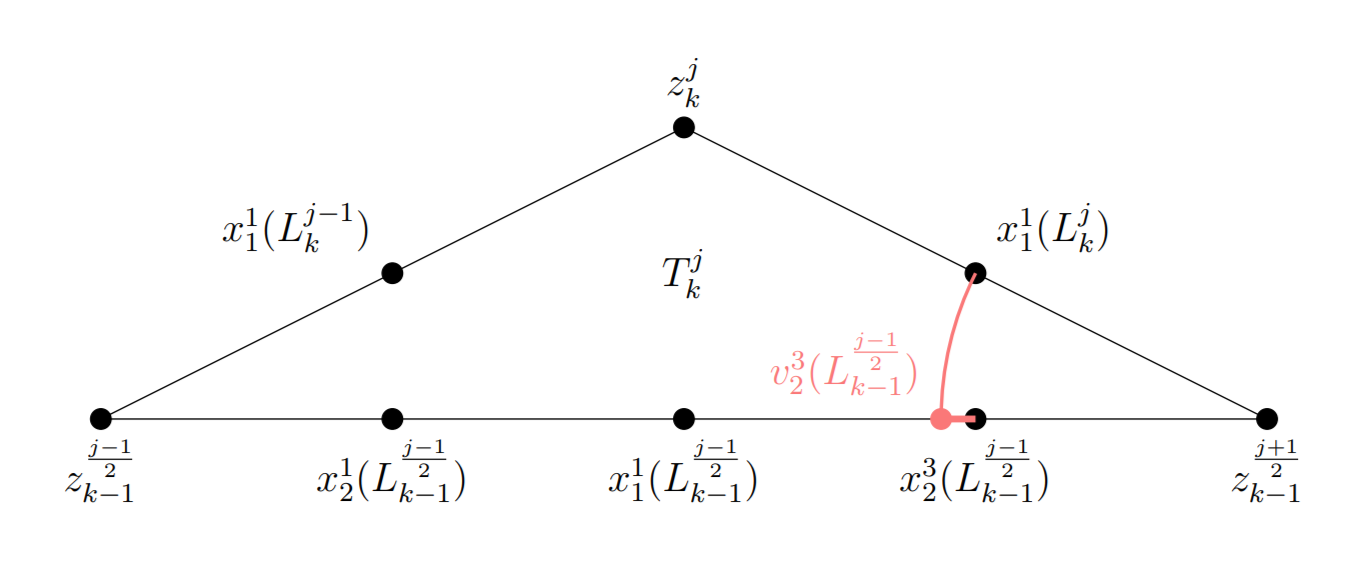}
			\caption{The case where $n>2^m$}
			\label{fig: Triangle lemma n>2^m}
			
		\end{center}
	\end{figure}
	Suppose that $n>2^{m-1}$, and let $d=n-2^{m-1}$. Let $x_{m-1}^d(L_k^j)$ be as in \eqref{e: dyadic points} and let $v_m^{d}(L_{k-1}^{\frac{j-1}{2}})$ be the point in $L_{k-1}^{\frac{j-1}{2}}$ such that 
	\begin{equation*}
		\left|v_m^{d}(L_{k-1}^{\frac{j-1}{2}})-z_k^{j+1} \right|=\left(1-\frac{d}{2^{m-1}}\right)s_k^j,
	\end{equation*} 
	and observe that
	\begin{equation*}
		\left|x_{m-1}^d(L_k^j)-z_k^{j+1}\right|=\left(1-\frac{d}{2^{m-1}}\right)s_k^j. 
	\end{equation*}
	
	Then, 
	\begin{align*}
		|x_{m-1}^{d}(L_{k}^{j})-v_m^{d}(L_{k-1}^{\frac{j-1}{2}})|&\leq \mathcal{H}^1(\arc_{z_k^{j+1}, (1-\frac{d}{2^{m-1}})s_k^j)}(x_{m-1}^{d}(L_{k}^{j}),v_m^{d}(L_{k-1}^{\frac{j-1}{2}})))\\
		&< (1-\frac{d}{2^{m-1}})s_{k}^j\theta_{k}^j\\
		&< (1-\frac{d}{2^{m-1}})s_{k}^j(C_1(s_{k-1}^j)^{\alpha/2})\\
		&=C_1(1-\frac{d}{2^{m-1}})s_{k}^j(s_{k-1}^j)^{\alpha/2}, 
	\end{align*}
	where the second to last inequality follows from (\ref{eq: angles for triangle }).
	Furthermore, 
	\begin{align*}
		|v_m^{d}(L_{k-1}^{\frac{j-1}{2}})-x_{m}^{n}(L_{k-1}^{\frac{j-1}{2}})|
		&=\left(1-\frac{d}{2^{m-1}}\right)s_k^j-\left(\frac{1}{2}-\frac{d}{2^{m}}\right)s_{k-1}^j\\
		&=\left(1-\frac{d}{2^{m-1}}\right)\left|s_{k}^j-\frac{s_{k-1}^j}{2} \right|\\
		&<\left(1-\frac{d}{2^{m-1}}\right)C_0^2(s_{k-1}^j)^{\alpha+1}, 
	\end{align*}
	where the last inequality follows from (\ref{eq: side length relationship on triangle}). Thus from the triangle inequality we have
	\begin{equation*}
		|x_{m}^{n}(L_{k-1}^{\frac{j-1}{2}})-x_{m-1}^{d}(L_{k}^{j})|<C_1\left(1-\frac{d}{2^{m-1}}\right)s_{k}^j(s_{k-1}^j)^{\alpha/2}+\left(1-\frac{d}{2^{m-1}}\right)C_0^2(s_{k-1}^j)^{\alpha+1},
	\end{equation*}
	which again with \eqref{e: uniform bound s_k} gives (\ref{eq:dyadic points close one step}). Thus, Lemma \ref{LEM: Triangle Lemma} is shown. 
     
 \end{proof}

\begin{theorem}[Dyadic Scales Theorem]\label{LEM: Dyadic Scales Lemma}
			Under the hypotheses of Lemma \ref{LEM: Main Lemma 1}, there exists $\rho_0\leq r_0$ such that for $x_0\in \Gamma$, $s_0\leq \rho_0$, and $m\in \mathbb{N}$ the following holds. For each $1\leq n\leq 2^m-1$ there exists a point $z_m^n\in \Gamma$ such that 
			\begin{equation}\label{e: goal of lemma 1}
				|x_m^n(L_0^0)-z_m^n|\leq Cs_0^{\alpha/2+1}, 
			\end{equation}
			where $L_0^0$ is as in the Half-Scales Lemma and $C$ is a constant depending only on $C_0$ and $C_1$. \\
		\end{theorem}
 
	To show Theorem \ref{LEM: Dyadic Scales Lemma} note that from the Half-Scales Lemma, the point $z_k^j$ is very close to $x_1^1(L_{k-1}^{\frac{j-1}{2}})$. So in order to show \eqref{e: goal of lemma 1}, a necessary next step is to iterate Lemma \ref{LEM: Triangle Lemma} to show that
	\begin{equation*}
		|x_k^j(L_0^0)-x_1^1(L_{k-1}^{\frac{j-1}{2}})|<C_2s_0^{\frac{\alpha}{2}+1}, 
	\end{equation*}
	for some constant $C_2$ depending only on $C_0$ and $C_1$. This means that for each dyadic point in $L_0^0$, we need to keep track of the family of dyadic points that satisfy \eqref{eq:dyadic points close one step} in each successive application of Lemma \ref{LEM: Triangle Lemma}, if we start with $x_k^j(L_0^0)$.  In other words, we need to track which side of the triangle contains the closest dyadic point in each iteration of Lemma \ref{LEM: Triangle Lemma}, to ensure that $x_k^j(L_0^0)$'s closest point after $k-1$ iterations is in fact on $L_{k-1}^{\frac{j-1}{2}}$. To do this, we need to introduce new notation.
	
	Denote $[n]_{m}:=n \mod(2^{m})$ and define 
	\begin{equation*}
		S(k,n,m):=\sum_{i=0}^{k-1}2^{k-(i+1)}\lambda(m,i),
	\end{equation*}
	where 
	\begin{equation*}
		\lambda(m,i)=\begin{cases} 
			0 &\text{if}\quad n\mod(2^{m-i})\leq 2^{m-(i+1)} \\
			1 &\text{if}\quad n\mod(2^{m-i})> 2^{m-(i+1)} .
		\end{cases}
	\end{equation*}
	The function $\lambda(m,i)=0$ if $x_{m-i}^{[n]_{m-i}}$ is closest to the left side of the triangle and $\lambda(m,i)=1$ if $x_{m-i}^{[n]_{m-i}}$ is closest to the right side of the triangle, and thus $S(k,n,m)$ gives the ``line address'' of the closest dyadic point to $x_m^n(L_0^0)$ after $k$ iterations.

	\begin{lemma}\label{LEM: Iterated Triangle Lemma (Step 3)}
		Assume the hypotheses of Lemma \ref{LEM: Main Lemma 1}. Fix $m\in \mathbb{N}$ and $1\leq n\leq 2^m-1$. Consider the point $x_m^n(L_0^0)$ in the dyadic collection $\{x_m^n(L_0^0)\}_{n=1}^{2^m-1}$ on $L_0^0$ such that $n$ is odd. For $1\leq k\leq m-1$, we have 
		\begin{equation*}
			\left|x_m^n(L_0^0)-x_{m-k}^{[n]_{m-k}}\left(L_k^{S(k,n,m)}\right)\right|\leq C_2s_0^{\alpha/2+1},
		\end{equation*}
		where $k$ is the number of iterations. \\
		In particular, 
		\begin{equation*}
			\left|x_m^n(L_0^0)-x_{1}^{1}\left(L_{m-1}^{S(m-1,n,m)}\right)\right|\leq C_2s_0^{\alpha/2+1}.
		\end{equation*}
	\end{lemma}

	Before proving Lemma \ref{LEM: Iterated Triangle Lemma (Step 3)}, we need to prove a counting lemma.
	\begin{lemma}\label{LEM: S(k,n,m) to S(k+1,n,m)} Fix $m\in \mathbb{N}$. For any $1\leq n \leq 2^m-1$ and $0\leq k< m-1$
		\begin{equation*}
			S(k+1,n,m)=2S(k,n,m)+\lambda(m,k),
		\end{equation*}
	\end{lemma}

 \begin{proof}[Proof of Lemma \ref{LEM: S(k,n,m) to S(k+1,n,m)}]
	\begin{align*}
		S(k+1,n,m)&=\sum_{i=0}^{k}2^{(k+1)-(i+1)}\lambda(m,i)\\
		&=2^{(k+1)-(k+1)}\lambda(m,k) +\sum_{i=0}^{k-1}2^{(k+1)-(i+1)}\lambda(m,i)\\
		&=\lambda(m,k) +2\sum_{i=0}^{k-1}2^{k-(i+1)}\lambda(m,i)\\
		&=\lambda(m,k) +2S(k,n,m).
	\end{align*}
 \end{proof}
	
 \begin{proof}[Proof of Lemma \ref{LEM: Iterated Triangle Lemma (Step 3)}]
 Fix $m\in 
	\mathbb{N}$ and $1\leq n\leq 2^m-1$. We first prove that 
	\begin{equation*}
		\left|x_m^n(L_0^0)-x_{m-k}^{[n]_{m-k}}\left(L_k^{S(k,n,m)}\right)\right|\leq C_1s_0^{\alpha/2+1}\sum_{i=1}^k\left(\frac{3}{4}\right)^{i+(i-1)\frac{\alpha}{2}}+C_0^2s_0^{\alpha+1}\sum_{i=1}^k\left(\frac{3}{4}\right)^{(i-1)(\alpha+1)}.
	\end{equation*}
	by induction on $k$. 
	
	The case of $k=1$ follows immediately from Lemma \ref{LEM: Triangle Lemma}.

	Now suppose that for some $k\in \mathbb{N}$ with $k<m-1$
	\begin{equation*}
		\left|x_m^n(L_0^0)-x_{m-k}^{[n]_{m-k}}\left(L_k^{S(k,n,m)}\right)\right|\leq C_1s_0^{\alpha/2+1}\sum_{i=1}^k\left(\frac{3}{4}\right)^{i+(i-1)\frac{\alpha}{2}}+C_0^2s_0^{\alpha+1}\sum_{i=1}^k\left(\frac{3}{4}\right)^{(i-1)(\alpha+1)}.
	\end{equation*}
	Applying Lemma \ref{LEM: Triangle Lemma} to $x_{m-k}^{[n]_{m-k}}\left(L_k^{S(k,n,m)}\right)$ with
	$$x_{m-k-1}^{[[n]_{m-k}]_{m-k-1}}\left(L_{k+1}^{2S(k,n,m)+\lambda(m,k)}\right)=x_{m-(k+1)}^{[n]_{m-k-1}}\left(L_{k+1}^{S(k+1,n,m)}\right)$$ 
	from Lemma \ref{LEM: S(k,n,m) to S(k+1,n,m)},
	we obtain 
	\begin{align*}
		\left|x_{m-k}^{[n]_{m-k}}\left(L_k^{S(k,n,m)}\right)-x_{m-(k+1)}^{[n]_{m-k-1}}\left(L_{k+1}^{S(k+1,n,m)}\right)\right|&\leq
		C_1s_0^{\alpha/2+1}\left(\frac{3}{4}\right)^{(k+1)+k\frac{\alpha}{2}}+C_0^2s_0^{\alpha+1}\left(\frac{3}{4}\right)^{k(\alpha+1)}. 
	\end{align*}
	Thus, the result follows from the triangle inequality 
	So for any $1\leq k< m-1$, 
	\begin{align*}
		\left|x_m^n(L_0^0)-x_{m-(k+1)}^{[n]_{m-(k+1)}}\left(L_{k+1}^{S(k+1,n,m)}\right)\right|&<
		C_1s_0^{\alpha/2+1}\sum_{i=1}^{k+1}\left(\frac{3}{4}\right)^{i+(i-1)\frac{\alpha}{2}}+C_0^2s_0^{\alpha+1}\sum_{i=1}^{k+1}\left(\frac{3}{4}\right)^{(i-1)(\alpha+1)}.
	\end{align*}
	Moreover, since both 
	\begin{equation*}
		\sum_{i=1}^{m-1}\left(\frac{3}{4}\right)^{i+(i-1)\frac{\alpha}{2}}\leq \sum_{i=1}^{\infty}\left(\frac{3}{4}\right)^{i+(i-1)\frac{\alpha}{2}}\leq 4 \quad \text{and} \quad \sum_{i=1}^{m-1}\left(\frac{3}{4}\right)^{(i-1)(\alpha+1)}\leq \sum_{i=1}^{m-1}\left(\frac{3}{4}\right)^{(i-1)(\alpha+1)}\leq 4, 
	\end{equation*}
	and $s_0^{\alpha+1}<s_0^{\alpha/2+1}$ since $s_0<1$, Lemma \ref{LEM: Iterated Triangle Lemma (Step 3)} follows. 
     
 \end{proof}

 \begin{proof}[Proof of Theorem \ref{LEM: Dyadic Scales Lemma}]
 Fix $m\in \mathbb{N}$ and let $1\leq n\leq 2^m-1$ odd. From Lemma \ref{LEM: Step 1} and Remark \ref{rmk: sk's} there exists a point $z_m^n\in \Gamma$ such that 
 \begin{equation*}
     |z_m^n-x_1^1(L_{m-1}^{\frac{n-1}{2}})|<C_0s_0^{\alpha/2+1}\left(\frac{3}{4}\right)^{m-1}.
 \end{equation*}
 The proof then follows from Lemma \ref{LEM: Iterated Triangle Lemma (Step 3)}, as long as $S(m-1,n,m)=\frac{n-1}{2}$. Indeed, if $[n]_{m-j}>2^{m-(j+1)}$, then if we write $n$ as the sum of powers of $2$, $2^{m-(j+1)}$ will be a term in that sum. Since 
     \begin{align*}
		      \lambda(m,j) 
		       &=\begin{cases} 
			      0 &\text{if}\quad [n]_{m-j}\leq 2^{m-(j+1)} \\
			       1 &\text{if}\quad [n]_{m-j}> 2^{m-(j+1)}, 
			   \end{cases}
		    \end{align*}
      we have
	    \begin{align*}
		        S(m-1, n,m)&=\sum_{j=0}^{m-2}2^{k-(j+1)}\lambda(m,j)\\
		        &=\sum_{\substack{j=0, \hdots, m-2:\\ 
  [n]_{m-j}>2^{m-1-(j+1)}}}2^{m-1-(j+1)}\\
		        &=\frac{1}{2}\sum_{\substack{j=0, \hdots, m-2:\\ 
  [n]_{m-j}>2^{m-1-(j+1)}}}2^{m-(j+1)}\\
		        &=\frac{1}{2}(n-1), 
		    \end{align*}
	    where the last equality follows from the fact that $n$ is odd.    
 \end{proof}
	
	We are now ready to prove Lemma \ref{LEM: Main Lemma 1}.

 \begin{proof}[Proof of Lemma \ref{LEM: Main Lemma 1}]

	Fix $m_0\in \mathbb{N}$ to be the smallest natural number such that $\frac{1}{2^{m_0}}<s_0^{\alpha/2}$. 
	Then, for any 
	$x\in L_0^0\cap B(x_0,s_0)$ between $x_0$ and $y$, there exists $n\in \mathbb{N}$ with $1\leq n \leq 2^{m_0}-1$, such that 
	\begin{equation*}
		|x-x_{m_0}^n(L_0^0)|< \frac{1}{2^{m_0}}s_0.
	\end{equation*}
	By Theorem \ref{LEM: Dyadic Scales Lemma}, there exists a point $z_{m_0}^n\in \Gamma $ such that 
	\begin{equation*}
		|x_{m_0}^n(L_0^0)-z_{m_0}^n|<Cs_0^{\alpha/2+1},
	\end{equation*}
	and then from the triangle inequality,
	\begin{align*}
		d(x,\Gamma)&\leq |x-x_{m_0}^n(L_0^0)|+|x_{m_0}^n(L_0^0)-z_{m_0}^n|\\
		&\leq \frac{1}{2^{m_0}}s_0+ Cs_0^{\alpha/2+1}\\
		&\leq Cs_0^{\alpha/2+1}.
	\end{align*}
	Thus we have shown that for any $x\in L_0^0\cap B(x_0,s_0)$ between $x_0$ and $y$, $d(x, \Gamma)<Cs_0^{\alpha/2+1}$. 
	
	Theorem \ref{LEM: Dyadic Scales Lemma} can be applied with $(L_0^0)^{\prime}$, the line through $x_0$ and $y^{\prime}\in \Gamma \cap \partial B(x_0,s_0)$, where $\mathcal{H}^1(\arc_{x_0,s_0}(y,y^{\prime}))\in [\pi s_0-C_0s_0^{\alpha+1}, \pi s_0]$. Since 
	\begin{equation*}
		|x_{m_0}^n((L_0^0)^{\prime})- x_{m_0}^n(L_0^0)|<2C_0s_0^{\alpha+1},
	\end{equation*}
	see Corollary \ref{COR: line choice}, it follows that for $x \in L_0^0 \cap B(x_0,s_0)$ between $x_0$ and $2x_0-y$, 
	\begin{align*}
		d(x,\Gamma)&\leq |x-x_{m_0}^n(L_0^0)|+|x_{m_0}^n(L_0^0)-x_{m_0}^n((L_0^0)^{\prime})|+|x_{m_0}^n((L_0^0)^{\prime})-\left(z_{m_0}^n\right)^{\prime}|\\
		&\leq \frac{1}{2^{m_0}}s_0+ 2C_0s_0^{\alpha+1}+ Cs_0^{\alpha/2+1}\\
		&\leq Cs_0^{\alpha/2+1}.
	\end{align*}
	
	We now show that $d(x, \Gamma\cap B(x_0,s_0))<Cs_0^{\alpha/2+1}$.
	Consider a point $x\in L_0^0\cap B(x_0,s_0-Cs_0^{\alpha/2+1})$, where we assume by translating that $x_0=0$. From our work above, there exists a point $z_{m_0}^n\in \Gamma$ such that
	\begin{equation*}
		|x-z_{m_0}^n|< Cs_0^{\alpha/2+1}.
	\end{equation*}
	Then, 
	\begin{equation*}
		|x_0-z_{m_0}^n|< s_0-Cs_0^{\alpha/2+1}+Cs_0^{\alpha/2+1}=s_0.
	\end{equation*}
	So for any $x \in L_0^0 \cap B(x_0,s_0-Cs_0^{\alpha/2+1})$, $d(x, \Gamma \cap B(x_0,s_0))<Cs_0^{\alpha/2+1}$. If $x\in L_0^0 \cap (B(x_0,s_0)\setminus B(x_0, s_0-Cs_0^{\alpha/2+1}))$, let $\widehat{x}:=\frac{x}{|x|}(s_0-Cs_0^{\alpha/2+1})$. Observe that $\widehat{x}\in B(x_0, s_0-Cs_0^{\alpha/2+1})$. Then,
	\begin{align*}
		|x-z_{m_0}^n|&\leq |x-\widehat{x}|+|\widehat{x}-z_{m_0}^n|\\
		&=|x|\left(1- \frac{1}{|x|}(s_0-Cs_0^{\alpha/2+1})\right)+|\widehat{x}-z_{m_0}^n|\\
		&= |x|\left(1- \frac{s_0}{|x|}+\frac{Cs_0^{\alpha/2+1}}{|x|})\right)+|\widehat{x}-z_{m_0}^n|\\
		&< |x|\left(\frac{Cs_0^{\alpha/2+1}}{|x|})\right)+|\widehat{x}-z_{m_0}^n|\\
		&<Cs_0^{\alpha/2+1}+Cs_0^{\alpha/2+1}\\
		&=2Cs_0^{\alpha/2+1}.
	\end{align*}
Our proof of Lemma \ref{LEM: Main Lemma 1} is complete.   \end{proof}

\section{Proof of Lemma \ref{LEM: Main Lemma 2}}\label{section: proof of lemma 2}
	Our goal in this section is to prove Lemma \ref{LEM: Main Lemma 2}, which ensures that the following picture holds for each $x\in \Gamma$ and $s_0$ small enough.
	\begin{figure}[H]
		\begin{center}
			\begin{tikzpicture}[scale=1]
				\def\angl{7}
				
				\coordinate (X1) at (0,0); 
				\filldraw[blue] (X1) circle (2pt);
				\node[below=2] at (X1) {$x$};
				
				\def\shft{1}
				\def\eps{.6}
				
				\draw[<->] ({4*\shft},\eps) to node[midway, right] {$C_0s_0^{\alpha/16+1}$}({4*\shft},0);

				\draw ({4*\shft},0)--({-4*\shft},0); 
				\draw[dashed] ({4*\shft},\eps)--({-4*\shft},\eps); 
				\draw[dashed]  ({4*\shft},-\eps)--({-4*\shft},-\eps); 
				
				\draw[ultra thick] (0,0) circle (3);
				\node[below right] at (-5,0) {$L_0^0$}; 
				
				\node[blue] at (-4.5,{1.2*\eps}) {\Large $\Gamma$};
				\coordinate (A) at (-4.2,{1.4*\eps}); 
				\def\a{-90}; 
				\coordinate (B) at (-3.8,{.5*\eps});
				\def\b{-45};
				\coordinate (C) at (-3.2,{-.7*\eps});
				\def\c{0};
				\coordinate (D) at (-2.8,{.2*\eps});
				\def\d{0};
				\coordinate (E) at (-1,{-.6*\eps});
				\def\e{40};
				\coordinate (F) at (-1.2,{.4*\eps});
				\def\f{210};
				\coordinate (G) at (-.5,{-.2*\eps});
				\def\g{-50};
				\coordinate (H) at (0,{0*\eps});
				\def\h{30};
				\coordinate (I) at (.4,{-.1*\eps});
				\def\i{20};
				\coordinate (J) at (1.3,{.3*\eps});
				\def\j{10};
				
				\coordinate (K) at (1.6,{0*\eps});
				\def\k{-90};
				\coordinate (L) at (1.5,{-.6*\eps});
				\def\l{-90};
				\coordinate (M) at (2,{-.8*\eps});
				\def\m{15};
				\coordinate (N) at (2.8,{.3*\eps});
				\def\n{0};
				\coordinate (O) at (3.3,{.05*\eps});
				\def\o{10};
				
				\draw[very thick,blue] 
				(A) to [out=\a,in={\b-180}]
				(B) to [out=\b,in={\c-180}]
				(C) to [out=\c,in={\d-180}]
				(D) to [out=\d,in={\e-180}] 
				(E) to [out=\e,in={\f-180}] 
				(F) to [out=\f-180,in={\g-180}] 
				(G) to [out=\g,in={\h-180}] 
				(H) to [out=\h,in={\i-180}] 
				(I) to [out=\i,in={\j-180}] 
				(J) to [out=\j,in={\k-180}] 
				(K) to [out=\k,in={\l-180}] 
				(L) to [out=\l,in={\m-180}] 
				(M) to [out=\m,in={\n-180}] 
				(N) to [out=\n,in={\o-180}] 
				(O);
				
				
				\fill[color=lavender] 
				(12:3) arc[start angle=12, end angle=168, radius=3]
				-- (12:3) -- cycle;

				\fill[color=spring] 
				(192:3) arc[start angle=192, end angle=348, radius=3]
				-- (192:3) -- cycle;
     \draw (3,2.9) -- (5.1,2.9) -- (5.1,1.7)--(3,1.7)--(3,2.9);
	\filldraw[lavender] (3.1,2.8) -- (3.5,2.8) -- (3.5,2.4)--(3.1,2.4)--(3.1,2.8);
	\node at (4.3,2.6) {$\Omega^+$};
	\filldraw[spring] (3.1,2.2) -- (3.5,2.2) -- (3.5,1.8)--(3.1,1.8)--(3.1,2.2);
	\node at (4.3,2) {$\Omega^-$};
    
			\end{tikzpicture}
            \caption{}
			\label{fig: main lemma 2 assume claim}
		\end{center}
	\end{figure}

\begin{proof}[Proof of Lemma \ref{LEM: Main Lemma 2}] Let $x_0\in \Gamma$ and let $s_0<\rho_1$, where $\rho_1\leq \rho_0$ to be chosen later. Suppose for contradiction that there exists a point $u\in \left(\Gamma \cap B(x_0,s_0)\right)\setminus (L_0^0)_{C_0s_0^{\alpha/16+1}}$. Observe that 
 \begin{equation}\label{eq: s bound}
     s:=|x_0-u|\in (C_0s_0^{\alpha/16+1}, s_0)
 \end{equation}
 and let $L_s:=L_s(x_0,u)$ denote the line through $x_0$ and $u$. For $s_0<\left(\frac{1}{2}\right)^{\frac{16}{3\alpha}}$, 
 \begin{equation}\label{e: first s0 of lemma 2}
     B(u, C_0s_0^{\alpha/4+1})\subset \mathbb{R}^2_+\setminus \overline{(L_0)_{C_0s_0^{\alpha/2+1}}}.
 \end{equation}
	
	\begin{figure}[H]
   	    \begin{center}
    \begin{tikzpicture}[scale=1]
	\def\angl{7}
    \coordinate (X1) at (0,0); 
	 
    \node (P) at \polar{5}{40} {}; 
    \node (Q) at \polar{5}{220} {}; 
    \draw[ultra thick, black] (P)--(Q);
    \node[orange] (Y) at \polar{2.4}{40} {};
    \node[below, black]  at (Y) {$u$};
    \draw (Y) circle (1);
    \filldraw[black] (Y) circle (2pt);
    \node[below right] at (P) {$L_s$}; 

    \node[above=4, black]  at \polar{4}{43} {$B(u,C_0s_0^{\alpha/4+1})$};

 \filldraw[black] (X) circle (2pt);
	 \node[below=2] at (X) {$x_0$};

\node (A) at \polar{1.4}{40} {};
\node[above left, black] at (A) {$v$};
    \filldraw[black] (A) circle (2pt);
    
	\def\shft{1}
	\def\eps{.6}

		
	\draw ({4*\shft},0)--({-4*\shft},0); 
	\draw[dashed] ({4*\shft},\eps)--({-4*\shft},\eps); 
	\draw[dashed]  ({4*\shft},-\eps)--({-4*\shft},-\eps);
 \draw[<->] ({4*\shft},\eps) to node[midway, right] {$C_0{s_0^{\alpha/16+1}}$}({4*\shft},0);
  \draw[<->] ({-4*\shft},\eps*0.5) to node[midway, left] {$C{s_0^{\alpha/2+1}}$}({-4*\shft},0);
 \draw[dashed, blue] ({4*\shft},\eps*0.5)--({-4*\shft},\eps*0.5); 
	\draw[dashed,blue]  ({4*\shft},-\eps*0.5)--({-4*\shft},-\eps*0.5); 
	
\draw[ultra thick] (0,0) circle (3);
\node[below right] at (-5,0) {$L_0^0$}; 
	\end{tikzpicture}
        
    \end{center}
   \caption{}
			\label{fig: set up for contradiction in lemma 2}
	\end{figure}

	Let $v$ denote the point on $L_s\cap \partial B(u,C_0s_0^{\alpha/4+1})$ between $u$ and $x_0$. Applying Lemma \ref{LEM: Main Lemma 1} to $L_s$, $x_0$, and $u$, we obtain the existence of a point $z\in \Gamma$ such that 
	\begin{equation}\label{e: height of z}
		|v-z|<Cs^{\alpha/2+1}. 
	\end{equation}
	Taking also $s_0<\left(\frac{C_0}{2C}\right)^{4/\alpha}$, it follows from the triangle inequality that 
	\begin{equation}\label{eq: t bound}
		t:=|u-z|\in \left(\frac{1}{2}C_0s_0^{\alpha/4+1},2C_0s_0^{\alpha/4+1}\right).
	\end{equation}

	The ball $B(u, t)$ contains $z$ on its boundary and for $s_0< \left(\frac{1}{4}\right)^{16/3\alpha} $ we have
    \begin{equation}\label{e: second s0 for lemma 2}
        B(u, t)\subset \mathbb{R}^2_+\setminus \overline{(L_0)_{C_0s_0^{\alpha/2+1}}}.
    \end{equation}

 	\begin{figure}[H]
		\begin{center}
   	    \begin{center}
    \begin{tikzpicture}[scale=1]
	\def\angl{7}
    \coordinate (X1) at (0,0); 
	 
    \node (P) at \polar{5}{40} {}; 
    \node (Q) at \polar{5}{220} {}; 
    \draw[ultra thick, black] (P)--(Q);
    \node[orange] (Y) at \polar{2.4}{40} {};
    \node[below, black]  at (Y) {$u$};
    \draw (Y) circle (1);
    \draw[red] (Y) circle (1.1);
    \filldraw[black] (Y) circle (2pt);
    \node[below right] at (P) {$L_s$}; 

    \node[above=4, red]  at \polar{4}{43} {$B(u,t)$};

 \filldraw[black] (X) circle (2pt);
	 \node[below=2] at (X) {$x_0$};

\node (A) at \polar{1.4}{40} {};
\node[below right, black] at (A) {$v$};
    \filldraw[black] (A) circle (2pt);

    \node (B) at \polar{1.35}{52} {};
\node[above left, black] at (B) {$z$};
    \filldraw[red] (B) circle (2pt);
    
	\def\shft{1}
	\def\eps{.6}

		
	\draw ({4*\shft},0)--({-4*\shft},0); 
	\draw[dashed] ({4*\shft},\eps)--({-4*\shft},\eps); 
	\draw[dashed]  ({4*\shft},-\eps)--({-4*\shft},-\eps);
 \draw[<->] ({4*\shft},\eps) to node[midway, right] {$C_0{s_0^{\alpha/16+1}}$}({4*\shft},0);
  \draw[<->] ({-4*\shft},\eps*0.5) to node[midway, left] {$C{s_0^{\alpha/2+1}}$}({-4*\shft},0);
 \draw[dashed, blue] ({4*\shft},\eps*0.5)--({-4*\shft},\eps*0.5); 
	\draw[dashed,blue]  ({4*\shft},-\eps*0.5)--({-4*\shft},-\eps*0.5); 
	
\draw[ultra thick] (0,0) circle (3);
\node[below right] at (-5,0) {$L_0^0$}; 
	\end{tikzpicture}
        
    \end{center}
   \caption{}
			\label{fig: finding z in lemma 2}
		\end{center}
	\end{figure}
	
	We first claim that 
    \begin{equation}\label{eq: all boundary close to Ls}
        \mathcal{H}^1\left(\arc_{u,t}(p,L_s)\right)\leq \left(2C_0t^{\alpha}+C^{\prime}s_0^{\alpha/4}\right)t \qquad \text{for all}\qquad p\in \Gamma\cap \partial B(u,t),
    \end{equation}
    where $C^{\prime}$ depends only on $C$ and $C_0$.
Let $p\in \Gamma \cap \partial B(u,t)$. From the Trapped Boundary Lemma, either $\mathcal{H}^1(\arc_{u,t}(p,z))\in [0,2C_0t^{\alpha+1}]$ or $\mathcal{H}^1(\arc_{u,t}(p,z))\in [\pi t-C_0t^{\alpha+1}, \pi t]$. If $\mathcal{H}^1(\arc_{u,t}(p,z))\in [0,2C_0t^{\alpha+1}]$ then 
\begin{equation}\label{eq: p to Ls prelim 1}
    \mathcal{H}^1\left(\arc_{u,t}(p,L_s)\right)\leq 2C_0t^{\alpha+1}+\mathcal{H}^1(\arc_{u,t}(z,L_s)). 
\end{equation}
If $\mathcal{H}^1(\arc_{u,t}(p,z))\in [\pi t-C_0t^{\alpha+1}, \pi t]$, let $z^{\prime}=2u-z$ be the antipodal point to $z$ on $\partial B(u,t)$. Then $\mathcal{H}^1(\arc_{u,t}(p,z^{\prime}))\in [0,C_0t^{\alpha+1}]$, so as above, 
\begin{equation}\label{eq: p to Ls prelim 2}
    \mathcal{H}^1\left(\arc_{u,t}(p,L_s)\right)\leq C_0t^{\alpha+1}+\mathcal{H}^1(\arc_{u,t}(z^{\prime},L_s)).
\end{equation}
From \eqref{eq: p to Ls prelim 1}, \eqref{eq: p to Ls prelim 2}, and the fact that $\mathcal{H}^1(\arc_{u,t}(z,L_s))=\mathcal{H}^1(\arc_{u,t}(z^{\prime},L_s))$, in order to prove \eqref{eq: all boundary close to Ls} it is sufficient to show that   
 \begin{equation}\label{eq: arc z to Ls}
    \mathcal{H}^1(\arc_{u,t}(z,L_s))<C^{\prime}s_0^{\alpha/4}t, 
 \end{equation}
 where $C^{\prime}$ depends only on $C$ and $C_0$.
    Indeed, $\mathcal{H}^1(\arc_{u,t}(z,L_s))=\omega t$, where $\omega:=\measuredangle zux_0$, and \eqref{e: height of z} and \eqref{eq: t bound} yield
 \begin{equation}\label{eq: angle btwn zux_0 bound}
     \sin(\omega) <2CC_0^{-1}s_0^{\alpha/4}.
 \end{equation}
  Using the Taylor series for $\arcsin(x)$ with $s_0$ small enough gives \eqref{eq: arc z to Ls}. Thus \eqref{eq: all boundary close to Ls} holds.

 Define 
	\begin{equation*}
		A_s:=A(x_0, s-Cs_0^{\alpha/2+1}, s+Cs_0^{\alpha/2+1}), 
	\end{equation*}
	the annular region containing the circumference $\partial B(x_0,s)$ with radius $Cs_0^{\alpha/2+1}$. $\partial A_s$ intersects $\partial B(u,t)$ in four points, $b_i$, $i=1,2,3,4$, labeled counterclockwise such that 
 \begin{align*}
     b_1,b_2&\in \partial B(x_0,s-Cs_0^{\alpha/2+1})\cap \partial B(u,t),\\
     b_3,b_4&\in \partial B(x_0,s+Cs_0^{\alpha/2+1})\cap \partial B(u,t).
 \end{align*}
 See Figure \ref{fig: annulus As and points bi}. We claim that 
\begin{equation}\label{eq: boundary of ball inside annulus doesn't see boundary}
    \arc_{u,t}(b_1,b_4)\subset \Omega^+ \qquad \text{and} \qquad  \arc_{u,t}(b_2,b_3) \subset \Omega^-, 
\end{equation}
	or vice versa, where $\arc_{u,t}(b_1,b_4)$ and $\arc_{u,t}(b_2,b_3)$ are the minimal arcs between $b_1$ and $b_4$, and between $b_2$ and $b_3$, respectively, on $\partial B(u,t)$. Observe that $\arc_{u,t}(b_1,b_4),\,\arc_{u,t}(b_2,b_3)\subset A_s$. \\

  	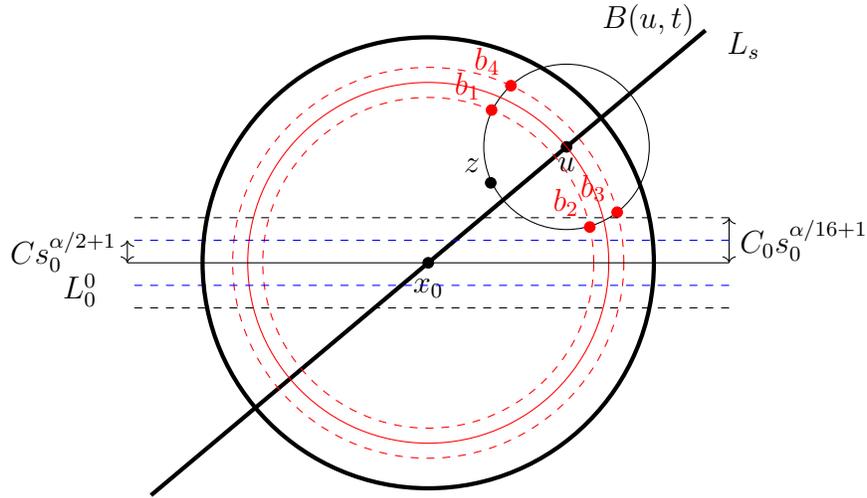
\begin{figure}[H]
		\begin{center}
    \begin{tikzpicture}[scale=1]
	\def\angl{7}
    \coordinate (X1) at (0,0); 
	 
    \node (P) at \polar{5}{40} {}; 
    \node (Q) at \polar{5}{220} {}; 
    \draw[ultra thick, black] (P)--(Q);
    \node[orange] (Y) at \polar{2.4}{40} {};
    \node[below, black]  at (Y) {$u$};
    \draw[black] (Y) circle (1.1);
    \filldraw[black] (Y) circle (2pt);
    \node[below right] at (P) {$L_s$}; 

    \node[above=4, black]  at \polar{4}{43} {$B(u,t)$};

 \filldraw[black] (X) circle (2pt);
	 \node[below=2] at (X) {$x_0$};

  \draw[red] (X) circle (2.4);
  \draw[dashed, red] (X) circle (2.6);
  \draw[dashed, red] (X) circle (2.2);
  
\node[orange] (B1) at \polar{2.2}{67.5} {};
    \node[above left, red]  at (B1) {$b_1$};
    \filldraw[red] (B1) circle (2pt);
    \node[orange] (B2) at \polar{2.2}{12.5} {};
    \node[above left, red]  at (B2) {$b_2$};
    \filldraw[red] (B2) circle (2pt);
        \node[orange] (B3) at \polar{2.6}{15} {};
    \node[above left, red]  at (B3) {$b_3$};
    \filldraw[red] (B3) circle (2pt);
    \node[orange] (B4) at \polar{2.6}{65} {};
    \node[above left, red]  at (B4) {$b_4$};
    \filldraw[red] (B4) circle (2pt);

    \node (B) at \polar{1.35}{52} {};
\node[above left, black] at (B) {$z$};
    \filldraw[black] (B) circle (2pt);
    
	\def\shft{1}
	\def\eps{.6}

		
	\draw ({4*\shft},0)--({-4*\shft},0); 
	\draw[dashed] ({4*\shft},\eps)--({-4*\shft},\eps); 
	\draw[dashed]  ({4*\shft},-\eps)--({-4*\shft},-\eps);
 \draw[<->] ({4*\shft},\eps) to node[midway, right] {$C_0{s_0^{\alpha/16+1}}$}({4*\shft},0);
  \draw[<->] ({-4*\shft},\eps*0.5) to node[midway, left] {$C{s_0^{\alpha/2+1}}$}({-4*\shft},0);
 \draw[dashed, blue] ({4*\shft},\eps*0.5)--({-4*\shft},\eps*0.5); 
	\draw[dashed,blue]  ({4*\shft},-\eps*0.5)--({-4*\shft},-\eps*0.5); 
	
\draw[ultra thick] (0,0) circle (3);
\node[below right] at (-5,0) {$L_0^0$}; 
	\end{tikzpicture}
        
    \end{center}
   \caption{The annulus $A_s$ and the points $b_i$}
			\label{fig: annulus As and points bi}
	\end{figure}
 To prove \eqref{eq: boundary of ball inside annulus doesn't see boundary}, we first show that $b_i\not\in \Gamma$ for all $i=1,2,3,4$. From \eqref{eq: all boundary close to Ls} it is sufficient to show 
	\begin{equation}\label{eq: b_is cant be in Gamma}
		\mathcal{H}^1(\arc_{u,t}(b_i,L_s))>\left( C^{\prime}s_0^{\alpha/4}+2C_0t^{\alpha}\right)t\qquad \text{for}\qquad i=1,2,3,4. 
	\end{equation}
    Observe that 
    \begin{equation}\label{eq: lengths of the b_i arcs}
        \begin{cases}
            \mathcal{H}^1(\arc_{u,t}(b_i, L_s))=\theta_{i}t &\text{if}\quad  0\leq \theta_i\leq \frac{\pi}{2},\\
            \mathcal{H}^1(\arc_{u,t}(b_i, L_s))=(\pi -\theta_{i})t &\text{if}\quad \frac{\pi}{2}\leq \theta_i \leq \pi,
        \end{cases}
    \end{equation}
	 where $\theta_{i}:= \measuredangle b_iu x_0$. 
	  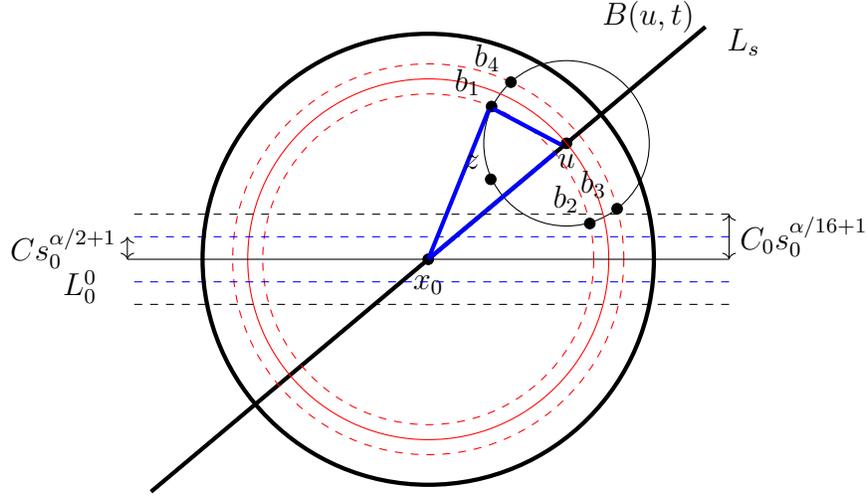
\begin{figure}[H]
		\begin{center}
    \begin{tikzpicture}[scale=1]
	\def\angl{7}
    \coordinate (X1) at (0,0); 
	 
    \node (P) at \polar{5}{40} {}; 
    \node (Q) at \polar{5}{220} {}; 
    \draw[ultra thick, black] (P)--(Q);
    \node[orange] (Y) at \polar{2.4}{40} {};
    \node[below, black]  at (Y) {$u$};
    \draw[black] (Y) circle (1.1);
    \filldraw[black] (Y) circle (2pt);
    \node[below right] at (P) {$L_s$}; 

    \node[above=4, black]  at \polar{4}{43} {$B(u,t)$};

 \filldraw[black] (X) circle (2pt);
	 \node[below=2] at (X) {$x_0$};

  \draw[red] (X) circle (2.4);
  \draw[dashed, red] (X) circle (2.6);
  \draw[dashed, red] (X) circle (2.2);
  
\node[orange] (B1) at \polar{2.2}{67.5} {};
    \node[above left, black]  at (B1) {$b_1$};
    \filldraw[black] (B1) circle (2pt);
    \node[orange] (B2) at \polar{2.2}{12.5} {};
    \node[above left, black]  at (B2) {$b_2$};
    \filldraw[black] (B2) circle (2pt);
        \node[orange] (B3) at \polar{2.6}{15} {};
    \node[above left, black]  at (B3) {$b_3$};
    \filldraw[black] (B3) circle (2pt);
    \node[orange] (B4) at \polar{2.6}{65} {};
    \node[above left, black]  at (B4) {$b_4$};
    \filldraw[black] (B4) circle (2pt);

\node[orange] (B11) at \polar{2.2}{71} {};
\node[orange] (Y1) at \polar{2.4}{37} {};
\node[orange] (B12) at \polar{2.3}{68} {};
\draw[ultra thick, blue] (X)--(B12);
\draw[ultra thick, blue] (B11)--(Y1);
\draw[ultra thick, blue] (Y)--(X);

    \node (B) at \polar{1.35}{52} {};
\node[above left, black] at (B) {$z$};
    \filldraw[black] (B) circle (2pt);
    
	\def\shft{1}
	\def\eps{.6}

		
	\draw ({4*\shft},0)--({-4*\shft},0); 
	\draw[dashed] ({4*\shft},\eps)--({-4*\shft},\eps); 
	\draw[dashed]  ({4*\shft},-\eps)--({-4*\shft},-\eps);
 \draw[<->] ({4*\shft},\eps) to node[midway, right] {$C_0{s_0^{\alpha/16+1}}$}({4*\shft},0);
  \draw[<->] ({-4*\shft},\eps*0.5) to node[midway, left] {$C{s_0^{\alpha/2+1}}$}({-4*\shft},0);
 \draw[dashed, blue] ({4*\shft},\eps*0.5)--({-4*\shft},\eps*0.5); 
	\draw[dashed,blue]  ({4*\shft},-\eps*0.5)--({-4*\shft},-\eps*0.5); 
	
\draw[ultra thick] (0,0) circle (3);
\node[below right] at (-5,0) {$L_0^0$}; 
	\end{tikzpicture}
        
    \end{center}
   \caption{Triangle with vertices $x_0$, $b_1$, and $u$}
			  \label{fig: triangle for bi}
	\end{figure}

	Consider the triangle with vertices $x_0$, $b_i$, and $u$, as seen in blue in Figure \ref{fig: triangle for bi}. We first show that \eqref{eq: b_is cant be in Gamma} holds for $i=1,2$. It follows from the law of cosines that
	\begin{align*}
		\cos(\theta_i)=\frac{s^2+t^2-\left(s-Cs_0^{\alpha/2+1}\right)^2}{2st}=\frac{t^2+2Css_0^{\alpha/2+1}-C^2s_0^{\alpha+2}}{2st}. 
	\end{align*}
If $0\leq \theta_i\leq \pi/2$, then \eqref{eq: s bound} and \eqref{eq: t bound} yield
	\begin{equation*}
		0\leq \cos(\theta_{i})\leq C_3s_0^{3\alpha/16}, 
	\end{equation*}
 where $C_3$ depends only on $C_0$ and $C$. If $\pi/2\leq \theta_i\leq \pi$, then $0\leq \pi-\theta_i\leq \pi/2$, and
 \eqref{eq: s bound} and \eqref{eq: t bound} yield
	\begin{equation*}
		0\leq \cos(\pi-\theta_{i})\leq C_3s_0^{3\alpha/16}. 
	\end{equation*}
Observe that this is equivalent to 
 \begin{equation*}
 \begin{cases}
     0<\sin(\frac{\pi}{2}-\theta_{i})<C_3s_0^{3\alpha/16} &\text{if}\quad 0\leq \theta_i\leq \frac{\pi}{2},\\
      0<\sin(\theta_{i}-\frac{\pi}{2})<C_3s_0^{3\alpha/16} &\text{if}\quad \frac{\pi}{2}\leq \theta_i\leq \pi.
 \end{cases} 
	\end{equation*}
As in \eqref{eq: angle btwn zux_0 bound}, taking $s_0$ small enough, for $i=1,2$, we obtain
\begin{equation}\label{eq: theta_is close to pi/2}
    \begin{cases}
         \frac{\pi}{2}-C_3s_0^{3\alpha/16}\leq \theta_{i}\leq \frac{\pi}{2} &\text{if}\quad 0\leq \theta_i\leq \frac{\pi}{2},\\
      \frac{\pi}{2}-C_3s_0^{3\alpha/16}\leq \pi -\theta_{i}\leq \frac{\pi}{2} &\text{if}\quad \frac{\pi}{2}\leq \theta_i\leq \pi,
    \end{cases}
\end{equation}
 where $C_3$ is different from above but still depends only on $C$ and $C_0$. 

For $i=3,4$, since 
\begin{equation}\label{eq: theta12 smaller than theta34}
    \cos(\theta_i)=\frac{s^2+t^2-\left(s+Cs_0^{\alpha/2+1}\right)^2}{2st}\leq \frac{s^2+t^2-\left(s-Cs_0^{\alpha/2+1}\right)^2}{2st}, 
\end{equation}
\eqref{eq: theta_is close to pi/2} also holds for $i=3,4$. Thus \eqref{eq: lengths of the b_i arcs} and \eqref{eq: theta_is close to pi/2} yield
\begin{equation*}
    \mathcal{H}^1(\arc_{u,t}(b_i,L_s))\geq \left( \frac{\pi}{2}-C_3s_0^{3\alpha/16}\right)t, 
\end{equation*}
and since 
\begin{equation}\label{e: third s0 lemma 2}
    \frac{\pi}{2}-C_3s_0^{3\alpha/16}\geq C^{\prime}s_0^{\alpha/4}+2C_0t^{\alpha}
\end{equation}
for $s_0$ small enough, \eqref{eq: b_is cant be in Gamma} holds. Thus $b_i\not\in \Gamma$ for all $i=1,2,3,4$.

We now prove \eqref{eq: boundary of ball inside annulus doesn't see boundary}. Since $b_1 \not \in \Gamma$, suppose without loss of generality that $b_1\in \Omega^+$. It then follows immediately that $b_2\in \Omega^-$ (see Figure \ref{fig: triangle for bi}). Since $b_4$ is on the same side of $L_s$ as $b_1$, we first claim that $b_4\in \Omega^+$ (and thus by similar argument $b_3 \in \Omega^-$). If not, then there exists a point $r\in \arc_{u,t}(b_1,b_4)\cap \Gamma$, where
\begin{align*}
    \mathcal{H}^1(\arc_{u,t}(r,L_s))\geq \mathcal{H}^1(\arc_{u,t}(b_i,L_s)),
\end{align*}
and from \eqref{eq: b_is cant be in Gamma}, this violates \eqref{eq: all boundary close to Ls}. Thus, $b_4\in \Omega^+$. Moreover this shows that $\arc_{u,t}(b_1,b_4)\cap \Gamma =\emptyset$. Similar arguments hold for $\arc_{u,t}(b_2,b_3)$, and thus \eqref{eq: boundary of ball inside annulus doesn't see boundary} holds. 

	\begin{figure}[H]
		\begin{center}
    \begin{tikzpicture}[scale=1]
	\def\angl{7}
    \coordinate (X1) at (0,0); 
	 
    \node (P) at \polar{5}{40} {}; 
    \node (Q) at \polar{5}{220} {}; 
    \draw[ultra thick, black] (P)--(Q);
    \node[orange] (Y) at \polar{2.4}{40} {};
    \node[below, black]  at (Y) {$u$};
    \draw[black] (Y) circle (1.1);
    \filldraw[black] (Y) circle (2pt);
    \node[below right] at (P) {$L_s$}; 

    \node[above=4, black]  at \polar{4}{43} {$B(u,t)$};

 \filldraw[black] (X) circle (2pt);
	 \node[below=2] at (X) {$x_0$};
     
      \filldraw[red] (2.4,0) circle (2pt);
	 \node[red, below left=2] at (2.4,0) {$q$};

  \draw[red] (X) circle (2.4);
  \draw[dashed, red] (X) circle (2.6);
  \draw[dashed, red] (X) circle (2.2);
  
\node[orange] (B1) at \polar{2.2}{67.5} {};
    \node[above left, black]  at (B1) {$b_1$};
    \filldraw[black] (B1) circle (1.5pt);
    \node[orange] (B2) at \polar{2.2}{12.5} {};
    \node[above left, black]  at (B2) {$b_2$};
    \filldraw[black] (B2) circle (1.5pt);
        \node[orange] (B3) at \polar{2.6}{15} {};
    \node[right, black]  at (B3) {$b_3$};
    \filldraw[black] (B3) circle (1.5pt);
    \node[orange] (B4) at \polar{2.6}{65} {};
    \node[above left, black]  at (B4) {$b_4$};
    \filldraw[black] (B4) circle (1.5pt);

    \draw[blue] (X) circle (2.5);
    \node[orange] (V) at \polar{2.5}{5} {};
    \node[below right, blue]  at (V) {$v$};
    \filldraw[blue] (V) circle (2pt);
    \node[orange] (Vstar) at \polar{2.5}{25} {};
    \node[above right, blue]  at (Vstar) {$v^*$};
    \filldraw[blue] (Vstar) circle (2pt);

	\def\shft{1}
	\def\eps{.6}

	\draw ({4*\shft},0)--({-4*\shft},0); 
	\draw[dashed] ({4*\shft},\eps)--({-4*\shft},\eps); 
	\draw[dashed]  ({4*\shft},-\eps)--({-4*\shft},-\eps);
 \draw[<->] ({4*\shft},\eps) to node[midway, right] {$C_0{s_0^{\alpha/16+1}}$}({4*\shft},0);
  \draw[<->] ({-4*\shft},\eps*0.5) to node[midway, left] {$C{s_0^{\alpha/2+1}}$}({-4*\shft},0);
 \draw[dashed, blue] ({4*\shft},\eps*0.5)--({-4*\shft},\eps*0.5); 
	\draw[dashed,blue]  ({4*\shft},-\eps*0.5)--({-4*\shft},-\eps*0.5); 
	
\draw[ultra thick] (0,0) circle (3);
\node[below right] at (-5,0) {$L_0^0$}; 
	\end{tikzpicture}
        
    \end{center}
   \caption{$\partial B(x_0,t^*)$}
			  \label{fig: points v and v*}
	\end{figure}
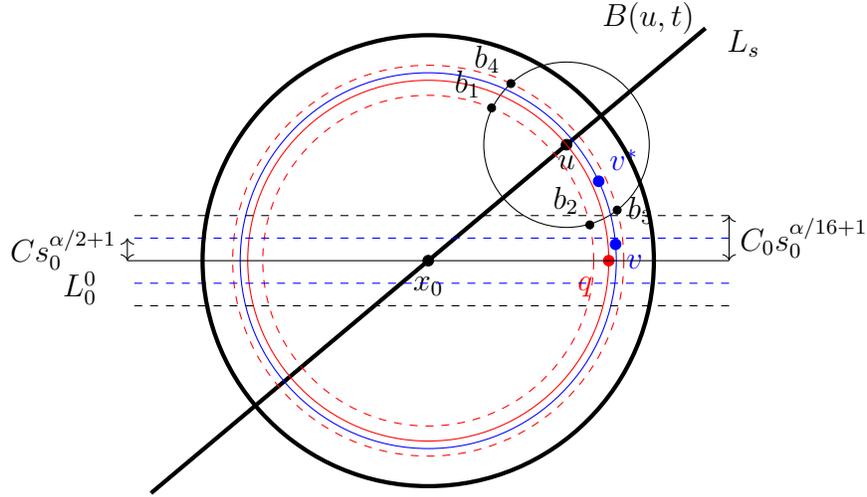

Let $q\in L_0^0\cap \partial B(x_0,s)$ such that $\mathcal{H}^1(arc_{x_0,s}(u,q))\leq \frac{\pi}{2}s$. It follows from Theorem \ref{LEM: Dyadic Scales Lemma} that there exists a point $v\in \Gamma \cap B(x_0,s_0)$ such that $|v-q|<Cs_0^{\alpha/2+1}$. In particular this means,
\begin{equation}\label{eq: t^* bound}
    \partial B(x_0,t^*)\subset A_s \qquad \text{where}\qquad t^*=|x_0-v|\in (s-Cs_0^{\alpha/2+1},s+Cs_0^{\alpha/2+1}), 
\end{equation}
and thus, 
\begin{equation*}
   \partial B(x_0,t^*)\cap \arc_{u,t}(b_1,b_4)\neq \emptyset\qquad \text{and}  \qquad  \partial B(x_0,t^*)\cap \arc_{u,t}(b_1,b_4)\neq \emptyset.
\end{equation*}
(See Figure \ref{fig: points v and v*} for $\partial B(x_0,t^*)$.) Since $\arc_{u,t}(b_1,b_4)$ and $\arc_{u,t}(b_2,b_3)$ are in complementary domains, there must exist a point $v^*\in A_s\cap B(u,t) \cap \partial B(x_0, t^*)$ by connectivity.

 To contradict the Trapped Boundary Lemma, we show that 
	\begin{equation}\label{eq: final contradiction}
		2C_0(t^*)^{\alpha+1}<\mathcal{H}^1(\arc_{x_0,t^*}(v,v^*))<\pi t^*-C_0(t^*)^{\alpha+1}. 
	\end{equation}
 From the triangle inequality,
 \begin{align}\label{e: triangle inequality final estimate}
 \begin{split}
      \mathcal{H}^1(\arc_{x_0,t^*}(v,v^*))&\leq \mathcal{H}^1(\arc_{x_0,t^*}(v,L_0^0))+\mathcal{H}^1(\arc_{x_0,t^*}(L_0^0,L_s))+\mathcal{H}^1(\arc_{x_0,t^*}(L_s,v^*)) \\
    \mathcal{H}^1(\arc_{x_0,t^*}(v,v^*))&\geq \mathcal{H}^1(\arc_{x_0,t^*}(L_0^0,L_s))-\mathcal{H}^1(\arc_{x_0,t^*}(L_s,v^*))-\mathcal{H}^1(\arc_{x_0,t^*}(v,L_0^0)).
 \end{split}
 \end{align}

Observe that $\mathcal{H}^1(\arc_{x_0,t^*}(L_0^0,L_s))=\theta_{s,0}t^*$ where $\theta_{s,0}=\measuredangle(L_s,L_0^0)$. Then from \eqref{eq: s bound}, 
\begin{equation*}
    \sin(\theta_{s,0})=\frac{d(u,L_0^0)}{s}>\frac{C_0s_0^{\alpha/16+1}}{s}>C_0s_0^{\alpha/16},
\end{equation*} 
and moreover, for $s_0$ small enough, $\theta_{s,0}\in (C_0s_0^{\alpha/16}, \frac{\pi}{2}]$. Thus,
 \begin{equation}\label{eq: angle between L00 and L_S}
     \mathcal{H}^1(\arc_{x_0,t^*}(L_0^0,L_s))\in (C_0s_0^{\alpha/16}t^*, \frac{\pi}{2}t^*].
 \end{equation}
We have $\mathcal{H}^1(\arc_{x_0,t^*}(L_s,v^*))=\theta_{v^*}t^*$, where $\theta_{v^*}=\measuredangle ux_0v^*$. Then \eqref{eq: t bound}, \eqref{eq: t^* bound}, and \eqref{eq: s bound} yield
  \begin{equation*}
      \sin(\theta_{v^*})=\frac{d(v^*,L_s)}{t^*}\leq \frac{t}{t^*}\leq C_4s_0^{3\alpha/16},
  \end{equation*} 
  and for $s_0$ small, we have $\theta_{v^*}\in [0,C_4s_0^{3\alpha/16})$, where again $C_4$ is possibly different from above but depends only on $C$ and $C_0$. Thus, 
 \begin{equation}\label{eq: angle between v* and L_s}
     \mathcal{H}^1(\arc_{x_0,t^*}(L_s,v^*))\in [0,C_4s_0^{3\alpha/16}t^*).
 \end{equation}
 Lastly, $\mathcal{H}^1(\arc_{x_0,t^*}(v,L_0^0))=\theta_{v}t^*$, where $\theta_{v}=\measuredangle vx_0q$. From choice of $v$, \eqref{eq: t^* bound}, and \eqref{eq: s bound}, we obtain
 \begin{equation*}
     \sin(\theta_v)=\frac{d(v,L_0^0)}{t^*}<\frac{Cs_0^{\alpha/2+1}}{t^*}\leq C_5 s_0^{7\alpha/16},
 \end{equation*}
where $C_5$ depends only on $C$ and $C_0$. Again, for  $s_0$ small enough we have
 \begin{equation}\label{eq: angle between v and L_0^0}
      \mathcal{H}^1(\arc_{x_0,t^*}(v,L_0^0))\in [0, C_5s_0^{7\alpha/16}t^*),
 \end{equation}
 where $C_5$ is different from above but still only depends on $C$ and $C_0$.
 Thus, \eqref{eq: angle between L00 and L_S}, \eqref{eq: angle between v* and L_s}, \eqref{eq: angle between v and L_0^0}, and \eqref{e: triangle inequality final estimate}, gives 
 \begin{equation}\label{e: fourth s0 in lemma 2}
      \mathcal{H}^1(\arc_{x_0,t^*}(v,v^*))<\left(\frac{\pi}{2}+C_4s_0^{3\alpha/16}+C_5s_0^{7\alpha/16}\right)t^*<\left(\frac{\pi}{2}+C_6s_0^{3\alpha/16}\right)t^*, 
 \end{equation}
 again taking $s_0$ small enough, where $C_6$ depends only on $C,C_0$. Recall from \eqref{eq: t^* bound} and \eqref{eq: s bound} that $t^*$ is small, and thus for $s_0$ small enough, 
 \begin{equation*}
     \left(\frac{\pi}{2}+C_6s_0^{3\alpha/16}\right)<\pi-C_0(t^*)^{\alpha}, 
 \end{equation*}
 i.e. the upper bound in \eqref{eq: final contradiction} holds. Similarly, \eqref{eq: angle between L00 and L_S}, \eqref{eq: angle between v* and L_s}, \eqref{eq: angle between v and L_0^0}, and \eqref{e: triangle inequality final estimate}, with $s_0$ small enough depending on $C_0,C$, and $\alpha$, 
 \begin{equation*}
     \mathcal{H}^1(\arc_{x_0,t^*}(v,v^*))>\left(C_0s_0^{\alpha/16} -4s_0^{3\alpha/16}-2C_5s_0^{7\alpha/16}\right)t^*>\frac{1}{2}C_0s_0^{\alpha/16}t^*.
 \end{equation*}
Recalling again that $t^*$ is small from \eqref{eq: t^* bound} and \eqref{eq: s bound},
 \begin{equation*}
     2C_0(t^*)^{\alpha}<\frac{1}{2}C_0s_0^{\alpha/16},
 \end{equation*}
 and thus \eqref{eq: final contradiction} holds. Choosing $\rho_1$ small enough so that \eqref{e: first s0 of lemma 2}, \eqref{eq: t bound}, \eqref{e: second s0 for lemma 2}, \eqref{eq: arc z to Ls}, \eqref{eq: theta_is close to pi/2}, \eqref{e: third s0 lemma 2}, \eqref{eq: angle between L00 and L_S}, \eqref{eq: angle between v and L_0^0}, and \eqref{e: fourth s0 in lemma 2} hold, completes the proof.
     
\end{proof}

\bibliographystyle{alpha}
\bibdata{references}
\bibliography{references}

\end{document}